%% file: RotationSets.tex
\documentclass[a4paper,11pt,twoside,reqno]{amsart}

\usepackage[english]{babel} 
\usepackage{inputenc, amsmath, amssymb, latexsym,
  epsfig, rotating, fancyhdr, amsthm, pifont, mathtools}
\usepackage{nicefrac}

\usepackage[figurename=Fig.]{caption}

\newcommand{\labelsize}{\scriptsize} % for modification of size of tikz labels

\usepackage{graphicx}
\usepackage{epstopdf}

\usepackage{tikz}
  \newlength\ticklength % to be set in each figure to modify the length of the ticks

\usepackage{pgfplots}
	\newlength\figureheight % to be set in each figure to ensure
	\newlength\figurewidth  % equal sizes of graphs within one figure
\usetikzlibrary{decorations.pathreplacing}

\definecolor{mygray}{rgb}{0.85,0.85,0.85}

\pgfplotsset{compat=1.8}

\input{bibstandard}

\newtheoremstyle{tobthm}{3pt}{3pt}{\itshape}{0pt}{\bfseries}{.}{0.5eM}{}
\theoremstyle{tobthm}

\newtheorem{definition}{Definition}[section]

\newtheorem{theorem}{Theorem}
\newtheorem{lemma}[theorem]{Lemma}
\newtheorem{corollary}[theorem]{Corollary}

\renewcommand{\epsilon}{\varepsilon}   % prettier epsilon

\newcommand{\tm}[1]{\mathbb{T}^{#1}}   % m-dimensional torus

\newcommand{\cm}[1]{\mathcal{C}_{#1}}  % set of lifts of cont. torus maps 
\newcommand{\hm}[1]{\mathcal{H}_{#1}}  % set of lifts of torus homeos
\newcommand{\rs}[1]{\varrho\left(#1\right)}      % rot. set of #1
\newcommand{\epsrs}[2]{\varrho^{#2}(#1)}    % eps-rot.-set of #1                                       

\newcommand{\dH}{\mathrm{d}_{\mathrm{H}}}     % d_H (Hausdorff metric)
\newcommand{\convH}{\rightarrow_{\mathrm{H}}} % -->_H (convergence w.r.t.
\newtheoremstyle{tobrem}{3pt}{3pt}{\normalfont}{0pt}{\bfseries}{.}{0.5em}{}
\theoremstyle{tobrem}

\newtheorem{rem}[definition]{Remark}

\numberwithin{equation}{section}
\numberwithin{figure}{section}

\title{\large \textsc{Set-oriented numerical computation of rotation sets}}

\author{ K.~Polotzek, K.~Padberg-Gehle and T.~J\"ager}

\begin{document}

\setlength{\abovedisplayskip}{1.0ex}
\setlength{\abovedisplayshortskip}{0.8ex}

\setlength{\belowdisplayskip}{1.0ex}
\setlength{\belowdisplayshortskip}{0.8ex}
\maketitle

\begin{abstract}
 \noindent
 We establish a set-oriented algorithm for the numerical approximation of the
 rotation set of homeomorphisms of the two-torus homotopic to the identity. A
 theoretical background is given by the concept of $\eps$-rotation sets. These
 are obtained by replacing orbits with $\eps$-pseudo-orbits in the definition of
 the Misiurewicz-Ziemian rotation set and are shown to converge to the latter as
 $\eps$ decreases to zero. Based on this result, we prove the convergence of the
 numerical approximations as precision and iteration time tend to
 infinity. Further, we provide analytic error estimates for the algorithm under
 an additional boundedness assumption, which is known to hold in many relevant
 cases and in particular for non-empty interior rotation sets.\medskip

\end{abstract}

\section{Introduction}

\noindent
Rotation theory for orientation-preserving homeomorphisms on the circle was
established by H. Poincar\'{e} in 1885 \cite{poincare:1885}, who showed that the
long-term behaviour of such maps can be classified by the rotation number. This
topological invariant provides a dichotomy for the dynamics depending on whether
it is a rational or an irrational number, corresponding to periodic or
quasiperiodic motion, respectively.

For homeomorphisms of higher dimensional tori, it is well-known that a unique
rotation vector does not need to exist anymore. In
\cite{MisiurewiczZiemian1989RotationSets}, Misiurewicz and Ziemian therefore
introduce the \textit{rotation set} of a torus homeomorphism, which collects all
possible asymptotic rotation vectors and carries strong information about the
system's asymptotic behaviour. For homeomorphisms on the two-torus, this compact
non-empty set is always convex \cite{MisiurewiczZiemian1989RotationSets}. If
it has non-empty interior, then the system has positive topological entropy
\cite{llibre/mackay:1991} and all rational points in the interior of the
rotation set are realised by periodic orbits \cite{franks:1989}. Apart from
these nowadays classical results, considerable progress has been made in the
last years on the rotation theory of surface homeomorphisms
(e.g.\ \cite{beguin/crovisier/leroux:2007}--\nocite{Jaeger2011EllipticStars,KoropeckiTal2012StrictlyToral,KoropeckiTal2012BoundedandUnbounded,Davalos2013SublinearDiffusion}\cite{LeCalvezZanata2015RationalModeLocking}),
including partial results on the well-known Franks-Misiurewicz Conjecture
\cite{franks/misiurewicz:1990}--\nocite{LeCalvezTal2015ForcingTheory,JaegerPasseggi2015SemiconjugateToIrrational,JaegerTal2016IrrationalRotationFactors,Kocsard2016MinimalTH,KoropeckiPasseggiSambarino2016FMC}\cite{AvilaLiuXu2017FMCcounterexamples}
that excludes the existence of certain line segments as rotation sets.

Concerning the shape of the rotation set, it is further known that generically
-- in the $\cC^0$-topology -- it is a polygon with rational vertices
\cite{Passeggi2013RationalRotationSets}. For generic area-preserving torus
homeomorphisms, this polygon is non-degenerate (has non-empty interior)
\cite{Guiheneuf2015RotationSetComputation}. Moreover, all rational polygons can
be realised as the rotation set of a torus homeomorphism
\cite{kwapisz:1992}. For a particular parameter family of such maps, derived,
through some elaborate inverse limit construction, from symbolic systems related
to beta expansions, bifurcations of the rotation set involving new types of
convex sets have been described in
\cite{BoylandDeCarvalhoHall2016NewRotationSets} (see also \cite{kwapisz:1995}).
Apart from that, however, there is still little knowledge concerning the
question which compact convex subsets of the plane may appear. Likewise, there
is still not much insight into the behaviour of the rotation set in `natural'
parameter families of torus diffeomorphisms, such as those studied by
theoretical physicists in
\cite{LeboeufKurchanFeingoldArovas1990PhaseSpaceLocalization}.

One serious obstruction for further progress in this direction is the fact that,
except for some particular cases, it is usually not possible to analytically
determine the rotation set. Moreover, due to the inherent nonlinearity of the
problem, the numerical computation has proven to be difficult as well.  In
\cite{Guiheneuf2015RotationSetComputation}, pointwise approaches to the
numerical approximation of the rotation set are discussed. Based on the
detection of periodic orbits of the system on the one hand, and of a
discretisation of the system by the projection to a finite lattice on the other
hand, the complexity of this numerical task as well as the restrictions of the
proposed approaches are discussed. It turns out that an accurate computation is
already difficult in the case where the rotation set is still a rational
polygon, but its vertices correspond to periodic orbits of larger periods. When
the rotation set is not a rational polygon at all, it has to be expected that
the situation is even worse.

Thus, our aim here is to establish a more reliable algorithm for the numerical
computation of rotation sets based on set-oriented methods. On the theoretical
level, this corresponds to considering $\epsilon$-pseudo-orbits reflecting the
inaccuracies of numerical calculations. This approach leads to the definition of
an $\epsilon$\textit{-rotation set} in Section~\ref{sec_epsrotset}. In the
two-dimensional case, we prove that as the perturbation size $\epsilon > 0$
decreases to zero the respective $\epsilon$-rotation sets converge to the
original rotation set (Theorem \ref{main}).

In Section~\ref{sec_algo}, we formulate our set-oriented algorithm for the
numerical approximation of the rotation set of a homeomorphism $f: \tm{2}
\longrightarrow \tm{2}$ homotopic to the identity, which is defined via a lift
$F: \mathbb{R}^2 \longrightarrow \mathbb{R}^2$ of $f$. Based on the fact that
the rotation set is approximated with arbitrary precision in the Hausdorff
metric by the normalised iterates $\frac{1}{n}F^n([0,1]^2)$ of the unit square
(Corollary~\ref{Col_Fn_convH}), our algorithm aims to provide a good
visualisation of the latter sets for large $n \in \mathbb{N}$. This is backed up
by rigorous error estimates and a result on the convergence of the
approximations to the true rotation set as the precision and iteration time tend
to infinity (Theorem~\ref{Th_box_err_asympt}). If the considered system has a
shadowing property, this leads to a further considerable improvement of the
results on the convergence rate (Theorem \ref{t.shadowing}). By a systematic
overestimation of the system's orbits, we further ensure that the numerical
algorithm always yields a super-set of the actual rotation set. This is
particularly important since, as mentioned, previous approaches tend to `miss
out' the vertices of polygonal rotation sets if these are realised by periodic
orbits of large period \cite{Guiheneuf2015RotationSetComputation}.

Numerical results for some parameter families inspired by
\cite{MisiurewiczZiemian1989RotationSets,%
  Guiheneuf2015RotationSetComputation,LeboeufKurchanFeingoldArovas1990PhaseSpaceLocalization}
are then presented in Section \ref{sec_nums}. It turns out that in most cases
the algorithm performs much better than predicted by the numerical error
estimates. The reason for this is possibly the fact that the considered systems
possess a shadowing property. As mentioned, this leads to a faster convergence
of the approximations.
\bigskip

\bigskip
\noindent{\bf Acknowledgements.} TJ has been supported by a Heisenberg
professorship of the German Research Council (grant OE 538/6-1). TJ and KPG
acknowledge funding from EU Marie-Sk{\l}odowska-Curie Innovative Training
Network {\em Critical Transitions in Complex Systems} (H2020-MSCA-2014-ITN
643073 CRITICS).

\clearpage
\section{Preliminaries} \label{sec_prelims}

\noindent
By $\conv(A)$ we denote the closed convex hull of a set $A\ssq\R^m$. The open
$\eps$-neighbourhood of points or sets in $\R^m$ will be denoted by $B_\eps(x)$,
respectively $B_\eps(A)$. The closure of a set $A$ is denoted by
$\overline{A}$. By $\dH(A,B)$ we denote the Hausdorff distance between two
subsets $A,B$ of a metric space and also write $A_n\convH A$ if a sequence
$(A_n)_{n\in\mathbb{N}}$ converges to $A$ in the Hausdorff topology as $n\to\infty$.  

For $m \in \mathbb{N}$, we
let $\tm{m} =\mathbb{R}^m / \mathbb{Z}^m$ denote the $m$-dimensional torus. The
set of all lifts $F:\R^m\to\R^m$ of continuous maps $f: \tm{m} \longrightarrow
\tm{m}$ homotopic to the identity is denoted by $\cm{m}$. The subset of $\cm{m}$
consisting of lifts of torus homeomorphisms is denoted by~$\hm{m}$. Note that
$\cm{m}$ consists of all continuous functions $F: \mathbb{R}^m \longrightarrow
\mathbb{R}^m$ that satisfy
\begin{align}
  F(x+k) \ = \ F(x) + k
  \label{Eaddintvect}
\end{align}
for all $x \in \mathbb{R}^m, k\in \mathbb{Z}^m$, and $\hm{m}$ is the set of all
such $F$ which are in addition homeomorphisms of the plane.  In their seminal
paper \cite{MisiurewiczZiemian1989RotationSets} Misiurewicz and Ziemian
introduced the rotation set as
%\begin{align*}
%  \rs{F} \coloneqq \left\{ v \in \mathbb{R}^2 \; \bigg| \; 
%  \exists (n_i)_{i\in\mathbb{N}} \text{ in } \mathbb{N}, 
%          (x_i)_{i\in\mathbb{N}} \text{ in } \mathbb{R}^2 : \,
%          n_i \to \infty, \frac{1}{n_i}\left(F^{n_i}(x_i)-x_i\right) \to v
%           (i\to\infty) \right\}.
%\end{align*}
%\begin{align*}
%  \rs{F} \coloneqq \left\{ v \in \mathbb{R}^2 \; \bigg| \; 
%  \exists (n_i)_{i\in\mathbb{N}} \text{ in } \mathbb{N}, 
%          (x_i)_{i\in\mathbb{N}} \text{ in } \mathbb{R}^2 : \,
%          n_i \to \infty, \frac{F^{n_i}(x_i)-x_i}{n_i} \to v
%           (i\to\infty) \right\}.
%\end{align*}
\begin{align*}
  \rs{F} \ = \ \left\{ v \in \mathbb{R}^m \; \Big| \; \exists n_i
    \nearrow\infty, x_i \in\mathbb{R}^m :\ \iLim \frac{F^{n_i}(x_i)-x_i}{n_i}
    \to v \right\}.
\end{align*}
Thus, $\rs{F}$ can be viewed as the collection of all possible rotation vectors
of the system. By writing
\begin{align}
  	K_k(F) \ = \ 
  	\left\{ \frac{F^k(x) - x}{k} \; \bigg| \; x \in [0,1]^m \right\}
  	\label{def_K_k}
\end{align}
for $k \in \mathbb{N}$, we alternatively obtain the rotation set as the upper
Hausdorff limit of the sequence $\left( K_n(F) \right)_{n \in \mathbb{N}}$, that
is,
\begin{align}
  	\rs{F} \ = \ \bigcap\limits_{n \geq 1} 
  	         \overline{\bigcup\limits_{k \geq n} K_k(F)} \ . 
  	\label{def_rot_set_sets}
\end{align}

\noindent
Moreover, in the two-dimensional case, the approximating sets $K_n(F)$ converge
to the rotation set in the Hausdorff distance.

\begin{lemma}[\cite{MisiurewiczZiemian1989RotationSets}] \label{main_Kn_convH} 
  Let $F \in \hm{2}$. Then
  %\begin{align*}
    $K_n(F) \convH \rs{F} \text{ as } n \to \infty.$
  %\end{align*}
\end{lemma}

\noindent
Furthermore, we have
\begin{theorem}[\cite{MisiurewiczZiemian1989RotationSets}] \label{Thm_convexity} Let $F \in \hm{2}$. Then the
  rotation set $\rs{F}$ is convex.
\end{theorem}
\begin{lemma}[\cite{MisiurewiczZiemian1989RotationSets}] \label{LProper3}
  Let $F \in \cm{m}$. Then $\rs{F} \subseteq \conv(K_n(F))$ for all $n\in\mathbb{N}$.
\end{lemma}

\noindent
A generalisation of Lemma \ref{LProper3} to $\epsilon$-rotation sets is given by
Lemma \ref{difficultest} below.  An important technical estimate in
\cite{MisiurewiczZiemian1989RotationSets} is the following.
\begin{lemma}[\cite{MisiurewiczZiemian1989RotationSets}]\label{l.MZ-estimate}
  Let $G\in\hm{2}$. Then $\conv \left( G([0,1]^2) \right) \subseteq
  \overline{B_{\sqrt{2}}(G([0,1]^2))}$.
\end{lemma}

\noindent
Applied to $G=F^n$ and taking into account that
\[
\dH \left( K_n(F), \frac{F^n([0,1]^2)}{n}  \right) \ <\ \frac{\sqrt{2}}{n}
\ ,
\]
 this immediately implies
the following consequence.
\begin{corollary}\label{c.MZ-estimate}
  If $F\in\hm{2}$ and $n\in\N$, then 
\[
   \conv(K_n(F))\ \ssq \ \overline{B_{\frac{3\sqrt{2}}{n}}(K_n(F))} \ .
\]
\end{corollary}

\section{The $\boldsymbol{\epsilon}$-rotation set} \label{sec_epsrotset}

\noindent
In order to model numerical approximations of rotation sets, we take into
account the fact that inaccuracies are inherent to computer
simulations. While the definition of the rotation set is based on proper orbits
of the underlying torus map, we introduce an alternative by allowing
perturbations of the system's orbits. This corresponds to the well-known concept
of $\epsilon$-pseudo-orbits and leads to the notion of $\eps$-rotation
sets. Since the start of this project, these have also been described
independently by Guiheneuf and Koropecki in
\cite{GuiheneufKoropecki2016RotationSetStability}.

%\marginpar{reason \mbox{"$\leq \epsilon$"?}}
Let $(X,\mathrm{d})$ be a metric space and $F: X \longrightarrow X$ a
self-map. For $n \in \mathbb{N}$ and $\epsilon \geq 0$, an $(n+1)$-tuple
$(\xi_j)_{j=0}^{n}$ of points $\xi_j \in X$ is called
an~$\epsilon$\textit{-pseudo-orbit} of length $n\!+\!1$ if
  \begin{align}
     \mathrm{d}\left(F(\xi_j),\xi_{j+1}\right) \ \leq \ \epsilon
     \label{def_eps_orbit}
  \end{align}
  for all $j \in \{0,\ldots,n-1\}$. In the same way infinite sequences
  $(\xi_j)_{j=0}^{\infty}$ or $(\xi_j)_{j=-\infty}^{\infty}$ satisfying
  (\ref{def_eps_orbit}) for all $j \in \mathbb{N}_0$ or $j\in\mathbb{Z}$,
  respectively, are called (infinite) $\epsilon$-pseudo-orbits. Obviously, every
  proper orbit is an $\eps$-pseudo-orbit for all $\eps>0$. Note that we slightly
  deviate from the standard definition by not requiring strictness of the
  inequality in (\ref{def_eps_orbit}). This will be very convenient later on for
  technical reasons, but has no significance on a conceptual level.

For $F \in \cm{m}$ and $\epsilon \geq 0$, we now define the
$\epsilon$\textit{-rotation set} to be the set of all accumulation points of
sequences of the form
\begin{align*}
  \left(\frac{\xi_{n_i}^i - \xi_0^i}{n_i}\right)_{i \in \mathbb{N}}
\end{align*}
where $n_i \to \infty$ as $i \to \infty$ and $(\xi^i_j)_{j=0}^{n_i}$ is an
$\epsilon$-pseudo-orbit of $F$ of length $n_i+1$ for each $i
\in \mathbb{N}$. Analogous to (\ref{def_rot_set_sets}), we let 
\begin{align*}
  K_k^\epsilon(F)  
  \ = \   \left\{ \frac{\xi_k-\xi_0}{k} \; \bigg| \;
                     \left( \xi_j \right)_{j=0}^{k} \text{ is an } 
                     \epsilon\text{-pseudo-orbit of } F 
                     \text{ with } \xi_0  \in [0,1]^m
             \right\}
  %\label{def_K_k_eps}
\end{align*}
for $k \in \mathbb{N}$ and can alternatively define the $\epsilon$-rotation set
by
\begin{align}
  \epsrs{F}{\epsilon} 
 \ = \ \bigcap\limits_{n \geq 1} 
    \overline{\bigcup\limits_{k \geq n} K_k^\epsilon(F)}\ . 
    \label{intersecteps}
\end{align}
Note that since $F$ commutes with integer translations, one could also replace
$\xi_0 \in [0,1]^m$ by $\xi_0\in\R^m$ in the definition of $K_k^\eps(F)$.
For every $k \in \mathbb{N}$, the inclusion $K_k(F) \subseteq K_k^{\epsilon}(F)$
is apparent, so that due to (\ref{def_rot_set_sets}) and
(\ref{intersecteps}), we obtain 
\begin{align}
  \rs{F} \ \subseteq \ \epsrs{F}{\epsilon}
  \label{Erset_subs_epsrs}
\end{align}
%\begin{align*}
%  \rs{F} 
%  =         \bigcap\limits_{n \geq 1} 
%            \overline{\bigcup\limits_{k \geq n} K_k(F)}
%  \subseteq \bigcap\limits_{n \geq 1} 
%            \overline{\bigcup\limits_{k \geq n} K_k^\epsilon(F)}
%  = \epsrs{F}{\epsilon}. 
%  %\label{Erset_subs_epsrs}
%\end{align*}
for all $\epsilon \geq 0$. We omit the elementary proof of the following lemma.
\begin{lemma}\label{props}
  Let $F \in \cm{m}, \epsilon > 0$ and $M = \max\nolimits_{x\in [0,1]^m} \|F(x)-x\|$. Then 
  \begin{itemize}
    \item[(i)] for each $k \in \mathbb{N}$ the set $K_k^{\epsilon}(F)$ is non-empty 
      and compact, and thus the same holds for  $\epsrs{F}{\epsilon}$;
    \item[(ii)] \label{props2} 
      $\rs{F} \subseteq \overline{B_M(0)}$ 
      and $\epsrs{F}{\epsilon} \subseteq \overline{B_{M+\epsilon}(0)}$.
  \end{itemize}
\end{lemma}

\noindent
We now aim to show convergence of the $\eps$-rotation sets as $\eps$ decreases
to zero. This presents the main result of this section. An alternative proof can
be found in \cite{GuiheneufKoropecki2016RotationSetStability}. However, we
include the proof both for the convenience of the reader and because the
employed arguments and estimates will become crucial again in
Section~\ref{sec_algo}.
 
 \clearpage
\begin{theorem} \label{main}
  Let $F \in \hm{2}$. Then 
  %\begin{align*}
    $\epsrs{F}{\epsilon} \convH \rs{F} \text{ as } \epsilon \to 0.$
  %\end{align*}
\end{theorem}

\noindent
Since for $\rs{F} \ssq \epsrs{F}{\epsilon}$ for all $\eps>0$ by
(\ref{Erset_subs_epsrs}), it remains to show that for every $\delta > 0$ there
is an $\epsilon > 0$ such that
\begin{align}
  \epsrs{F}{\epsilon} \ \subseteq \ B_\delta(\rs{F}) \ .
  \label{main2}
\end{align}
We split the proof into Lemmas~\ref{difficultest} and \ref{Prop_knconv},
beginning with a formulation of Lemma~\ref{LProper3} in terms of
$\epsilon$-rotation sets.
 
\begin{lemma} \label{difficultest}
  Let $F \in \cm{m}$ and $\epsilon \geq 0$. Then
    $\epsrs{F}{\epsilon} \subseteq \conv (K_n^{\epsilon}(F))$
    for all $n \in \mathbb{N}$.
\end{lemma}

\begin{proof}
  Let $k,n\in \mathbb{N}$. For an $\epsilon$-pseudo-orbit $(\xi_j)_{j=0}^{kn}$
  with $\frac{1}{kn} \left( \xi_{kn} -\xi_0 \right) \in K_{kn}^\epsilon(F)$
  we~have
  \begin{align*}
        \frac{1}{kn} \left( \xi_{kn} -\xi_0 \right)
    \ = \    \frac{1}{k} \sum\limits_{i=0}^{k-1}  
        \frac{1}{n} \left( \xi_{(i+1)n} -\xi_{in} \right)
    \ \in \  \conv \left( K_n^\epsilon(F) \right)\ ,
  \end{align*}
  since the tuples $( \xi_{in}, \xi_{in+1}, \ldots,\xi_{(i+1)n})$ are
  $\epsilon$-pseudo-orbits of $F$ of length $n+1$ and hence $\frac{1}{n}(
  \xi_{(i+1)n} -\xi_{in} ) \in K_n^{\epsilon}(F)$ for each $i \in
  \{0,\ldots,k-1\}$.  Consequently, for all $k,n \in \mathbb{N}$ we have
  \begin{align}
    \label{Econv_eps}
    K_{kn}^{\epsilon}(F) \ \subseteq \ \conv(K_n^{\epsilon}(F))\ .
  \end{align}
  Now let $k,n\in\mathbb{N}$ with $k\geq n$. Each number $k$ can uniquely be
  split into $k = m_k n + r_k$, where $m_k \in \mathbb{N}$ is the integer part
  of $\frac{k}{n}$ and $r_k \in \{0,\ldots,n-1\}$. If $(\xi_j)_{j=0}^ k$ is
  an~$\epsilon$-pseudo-orbit of length $k+1$, we find
  \begin{align}
    \label{rest}
    K_k^\epsilon(F) 
    \ \ni \ \frac{1}{k} \left( \xi_k - \xi_0 \right)
    \ = \   \frac{1}{k} \left( \xi_{m_k n+r_k} - \xi_{m_k n} \right)
       +\frac{1}{k} \left( \xi_{m_k n} - \xi_0 \right) \ .
  \end{align}
  Since the tuple $(\xi_{m_k n+j})_{j=0}^{r_k}$ is an $\epsilon$-pseudo-orbit of $F$
  of length $r_k+1$, we obtain~that
  \begin{eqnarray*}
  \lefteqn{\left\| \xi_{m_k n+r_k} - \xi_{m_k n} \right\|
   \ = \  \left\| \sum\limits_{i=0}^{r_k-1} 
    \xi_{m_k n+i+1} - \xi_{m_k n +i} \right\| } \\
  &\leq & \sum\limits_{i=0}^{r_k-1} \left( \left\| 
    F(\xi_{m_k n +i})-\xi_{m_k n +i} \right\| + \epsilon \right) 
    \ \leq\  r_k (M + \epsilon)  \ . 
 \end{eqnarray*}
 Thus, the norm of the first summand in (\ref{rest}) is bounded by
 $\frac{r_k}{k}(M+\epsilon)$. For the second summand, we have
  \begin{align*}
   \frac{1}{k}\left(\xi_{m_k}-\xi_0\right)
    \ = \ \frac{m_k n}{k} \cdot \frac{1}{m_k n} \left( \xi_{m_k n} - \xi_0 \right)
    \ \in \ \frac{m_k n}{k} \cdot K_{m_k n}^\epsilon(F) \ .
  \end{align*}    
  By (\ref{Econv_eps}), $K_{m_k n}^\epsilon(F)$ is included in $\conv \left( K_n^\epsilon(F) \right)$.
  Altogether, we obtain 
  \begin{align}
    K_k^\epsilon(F) &\ \subseteq \ \overline{B_{\frac{r_k}{k}(M +
        \epsilon)}\left(\frac{m_k n}{k} \cdot K_{m_k n}^\epsilon(F)\right)}
    \nonumber \\
    &\ \subseteq \ \overline{B_{\frac{r_k}{k}(M + \epsilon)}\left(\left( 1 -
          \frac{r_k}{k} \right)\conv\left( K_n^\epsilon(F) \right)\right)} \ .
    \label{diffproof1}
  \end{align}
  For a fixed value of $n$, the inclusion (\ref{diffproof1}) is valid for all $k
  \geq n$. Therefore, for each $n \in \mathbb{N}$, the representation
  (\ref{intersecteps}) of the $\epsilon$-rotation set yields
   \begin{align*}
        \epsrs{F}{\epsilon}
        & \ = \ \bigcap\limits_{l \geq 1}
        \overline{\bigcup\limits_{k \geq l} K_k^\epsilon(F)}
         \ \subseteq\ 
        \bigcap\limits_{l \geq n} \overline{\bigcup\limits_{k \geq l}
          K_k^\epsilon(F)} \\ 
          &\ \subseteq \ \bigcap\limits_{l \geq n}
        \overline{\bigcup\limits_{k \geq l} B_{\frac{r_k}{k}(M +
            \epsilon)}\left(\left(1 - \frac{r_k}{k} \right) \conv \left(
          K_n^\epsilon(F) \right)\right)} 
         \ = \ \conv \left( K_n^\epsilon(F) \right)\ ,
  \end{align*}  
  due to the convergence $\frac{r_k}{k} \to 0$ as $k \to \infty$. 
\end{proof}

\begin{lemma} \label{Prop_knconv}
  Let $F \in \cm{m}$ and $n \in \mathbb{N}$. Then
  %\begin{align*}
    $K_n^\epsilon(F)  \convH  K_n(F) \text{ as } \epsilon \to 0.$
  %\end{align*}
\end{lemma}

\begin{proof}
  The inclusion $K_n(F) \subseteq K_n^{\epsilon}(F)$ is evident for every
  $\epsilon > 0$. Let $\delta > 0$. Due to the continuity of $F$, there exists an
  $\epsilon_0 > 0$ such that $\|F^n(\xi_0)-\xi_n\| < n \delta$ for every
  $\epsilon$-pseudo-orbit $(\xi_j)_{j=0}^n$ of $F$ with $0<\epsilon \leq
  \epsilon_0$. Therefore, we obtain
  \begin{displaymath}
        \frac{\xi_n - \xi_0}{n}
    \ = \   \frac{F^n(\xi_0)-\xi_0}{n} + \frac{\xi_n-F^n(\xi_0)}{n}
    \ \in\  B_\delta(K_n(F)) \ . 
  \end{displaymath}
  As $\delta > 0$ was chosen arbitrarily and $n\in\N$ is fixed, this proves the
  convergence $K_n^{\epsilon}(F) \convH K_n(F)$ as $\epsilon \to 0$.
\end{proof}

\medskip\noindent {\em Proof of Theorem \ref{main}.}\quad Recall that in order
to prove Theorem \ref{main}, we have to show the inclusion (\ref{main2}). Let
$\delta > 0$. By Lemma \ref{main_Kn_convH}, the rotation set $\rs{F}$ can be
approximated up to an arbitrary precision with respect to the Hausdorff distance
by the sets $K_n(F)$. Therefore, we choose $n_0 \in \mathbb{N}$ sufficiently
large, so that
\begin{align}
  \label{proof1}
  K_{n_0}(F) \ \subseteq \ \overline{B_{\frac{\delta}{2}}(\rs{F})} \ .
\end{align}
Furthermore, by Lemma \ref{Prop_knconv} there exists an $\epsilon_0 > 0$
sufficiently small, such that
\begin{align}
  \label{proof2}
  K_{n_0}^{\epsilon}(F) \ \subseteq \ \overline{B_{\frac{\delta}{2}}(K_{n_0}(F))}
\end{align}
for all $0 < \epsilon \leq \epsilon_0$. Applying Lemma \ref{difficultest},
together with (\ref{proof1}) and (\ref{proof2}) we obtain
\begin{align*}
    \epsrs{F}{\epsilon} & 
  \ \subseteq \ 
    \conv \left( K_{n_0}^{\epsilon}(F) \right)
  \ \subseteq \ 
    \conv \left( B_{\frac{\delta}{2}}(K_{n_0}(F) \right) \\ & \ 
  \subseteq \ 
    \conv \left(B_\delta(\rs{F}) \right)
  \ = \ \overline{B_\delta(\rs{F})}
\end{align*}
for every $0 < \epsilon \leq \epsilon_0$. %For the last equality Note here that
$\overline{B_\delta(\rs{F})}$ is convex itself, since the rotation set is convex
by Theorem~\ref{Thm_convexity}. As $\delta > 0$ was arbitrary, this shows the
asserted convergence $\epsrs{F}{\epsilon} \convH \rs{F}$ as $\epsilon \to
0$. \qed

\section{Set-oriented computation of rotation sets} \label{sec_algo}

\subsection{Why set-oriented numerics -- shortcomings of the direct approach.}

Before we turn to our algorithm for the set-oriented computation of rotation
sets, we briefly want to comment on the problems that arise with a more
conventional approach. 

The easiest way to compute the rotation set of a torus
homeomorphism numerically would be to fix a standard grid $\Gamma\ssq[0,1]^2$ of
$N\times N$ points, to compute the normalised displacement vectors
$\frac{1}{n}(F^n(x)-x)$ for all $x\in\Gamma$ and some large $n\in\N$ and to plot the
collection of these vectors in order to obtain an approximation of the rotation
set. However, one may now consider the following situation, which actually turns
out to be generic (with respect to $\cC^0$-topology, see
\cite{Guiheneuf2015RotationSetComputation}): Suppose that an area-preserving
torus homeomorphism $f$ has a rotation set with non-empty interior, and that the
Lebesgue measure $\Leb_{\T^2}$ on $\torus$ is ergodic with respect to $f$. Then
there exists a well-defined rotation vector
\[
\varrho_f(\Leb_{\T^2}) \ = \ \int_{\T^2} \varphi(x) \ dx   ,
\]
where $\varphi(x) = F(x)-x$ is interpreted as a function on the torus, and we
have
\[
\nLim \frac{F^n(x)-x}{n} \ = \ \ntel \inergsum (\varphi\circ f^i)(x) \ = \ \varrho_f(\Leb_{\T^2})
\]
for Lebesgue-a.e.\ $x\in\torus$. Hence, if $n$ above is chosen too large, the
rotation vectors coming from starting points in the grid $\Gamma$ will almost
surely be arbitrarily close to $\varrho_f(\Leb_{\T^2})$. In this case, the
numerical approximation of the rotation set will show the singleton
$\varrho_f(\Leb_{\T^2})$, or something very close to it, whereas the true
rotation set is much bigger due to its non-empty interior. Figure
\ref{Fig_dir_mthd_a1b1} illustrates this numerical phenomenon for a standard
example $f_{1,1}$ (see Section \ref{sec_nums}) of a torus diffeomorphism given
by the lift
\begin{equation*} %\label{e.standard_exp}
  F_{1,1}:\mathbb{R}^2 \to\mathbb{R}^2 \ , \quad (x,y) \mapsto \big{(}x+\sin(2\pi(y+\sin(2\pi x))),\ y+\sin(2\pi x)\big{)},
\end{equation*}
which exhibits $\varrho_{f_{1,1}}(\Leb_{\T^2}) = (0,0)$ and is known to have the
rotation set $\rs{F_{1,1}} = [-1,1]^2$ (see Lemma \ref{l.rotset_F11}); as the
number of iterations increases the majority of the approximate rotation vectors
based on the iteration of grid points tend towards the origin and only the fixed
points at the vertices are detected.

\begin{figure}[ht]
  \labelsize
  \setlength\figureheight{6.3cm} 
  \setlength\figurewidth{6.3cm}
  \setlength{\ticklength}{1mm}
  \centering
  \begin{minipage}[t]{0.4\textwidth}
    \centering \vspace{0pt} % causes minipages being top-aligned
    \begin{tikzpicture} % n = 80
      \begin{scope}[scale=1]
        \begin{axis}[
          width=\figurewidth, height=\figureheight,
          enlargelimits=false,
          axis on top,
          xtick={-1,0,1}, ytick={-1,0,1}, 
          major tick length = \ticklength,
          x axis line style = {-}, y axis line style = {-}
          ]
          \addplot graphics [xmin=-1.1,xmax=1.1,ymin=-1.1,ymax=1.1]
                   {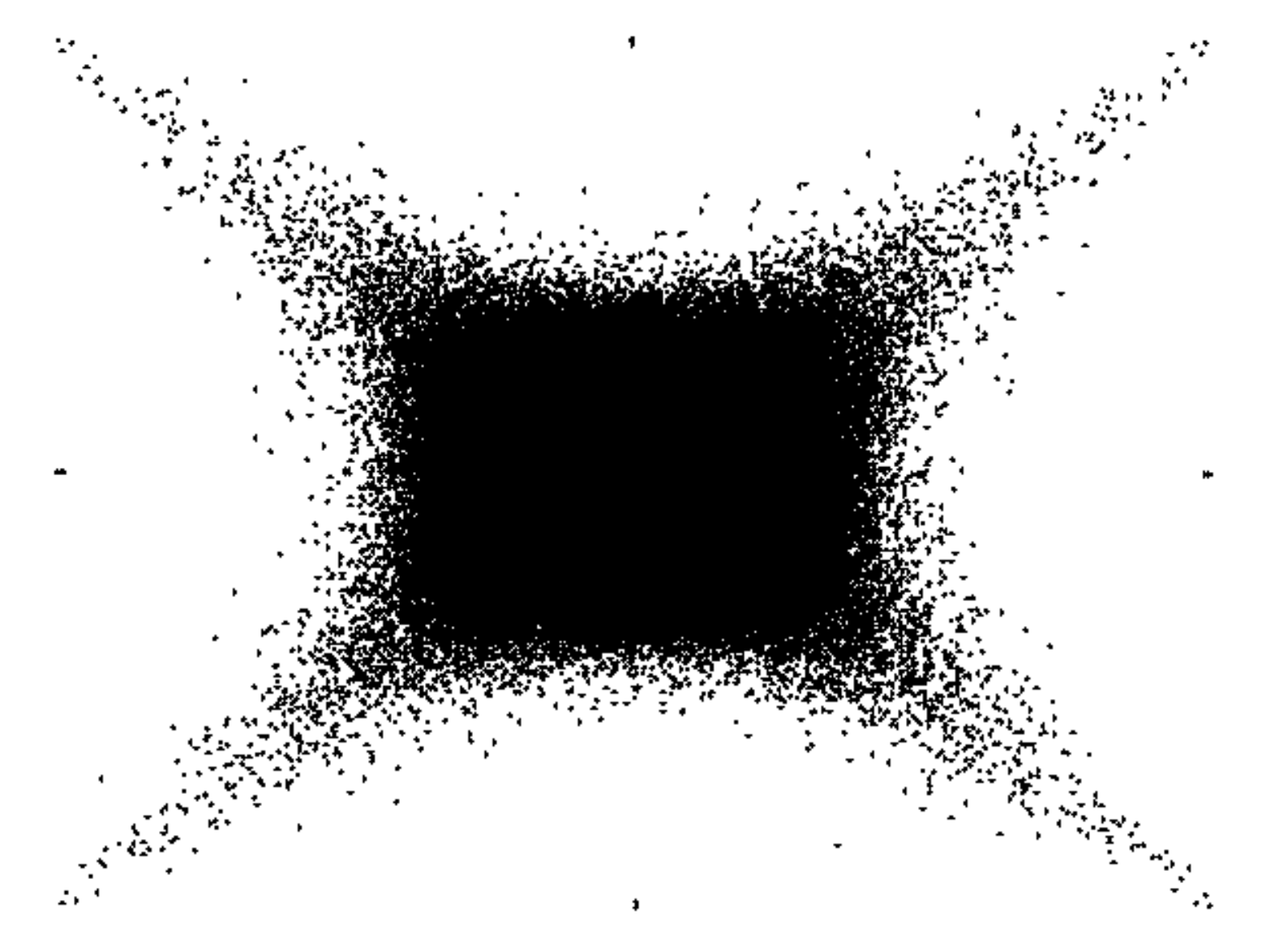};
          \draw[dashed, color=lightgray] 
            (axis cs:-1,-1) -- (axis cs:-1,1) -- 
            (axis cs: 1,1) -- (axis cs: 1,-1)-- 
            (axis cs:-1,-1);
        \end{axis}
      \end{scope}    
    \end{tikzpicture}
    %\label{Fig_dir_mthd_a1b1_80}
  \end{minipage}
  \hspace{1.5cm}
  \begin{minipage}[t]{0.4\textwidth}
    \centering \vspace{0pt} % causes minipages being top-aligned
    \begin{tikzpicture} % n = 2500
      \begin{scope}[scale=1]
        \begin{axis}[
          width=\figurewidth, height=\figureheight,
          enlargelimits=false,
          axis on top,
          xtick={-1,0,1}, ytick={-1,0,1}, 
          major tick length = \ticklength,
          x axis line style = {-}, y axis line style = {-},
          ]
          \addplot graphics [xmin=-1.1,xmax=1.1,ymin=-1.1,ymax=1.1]
                   {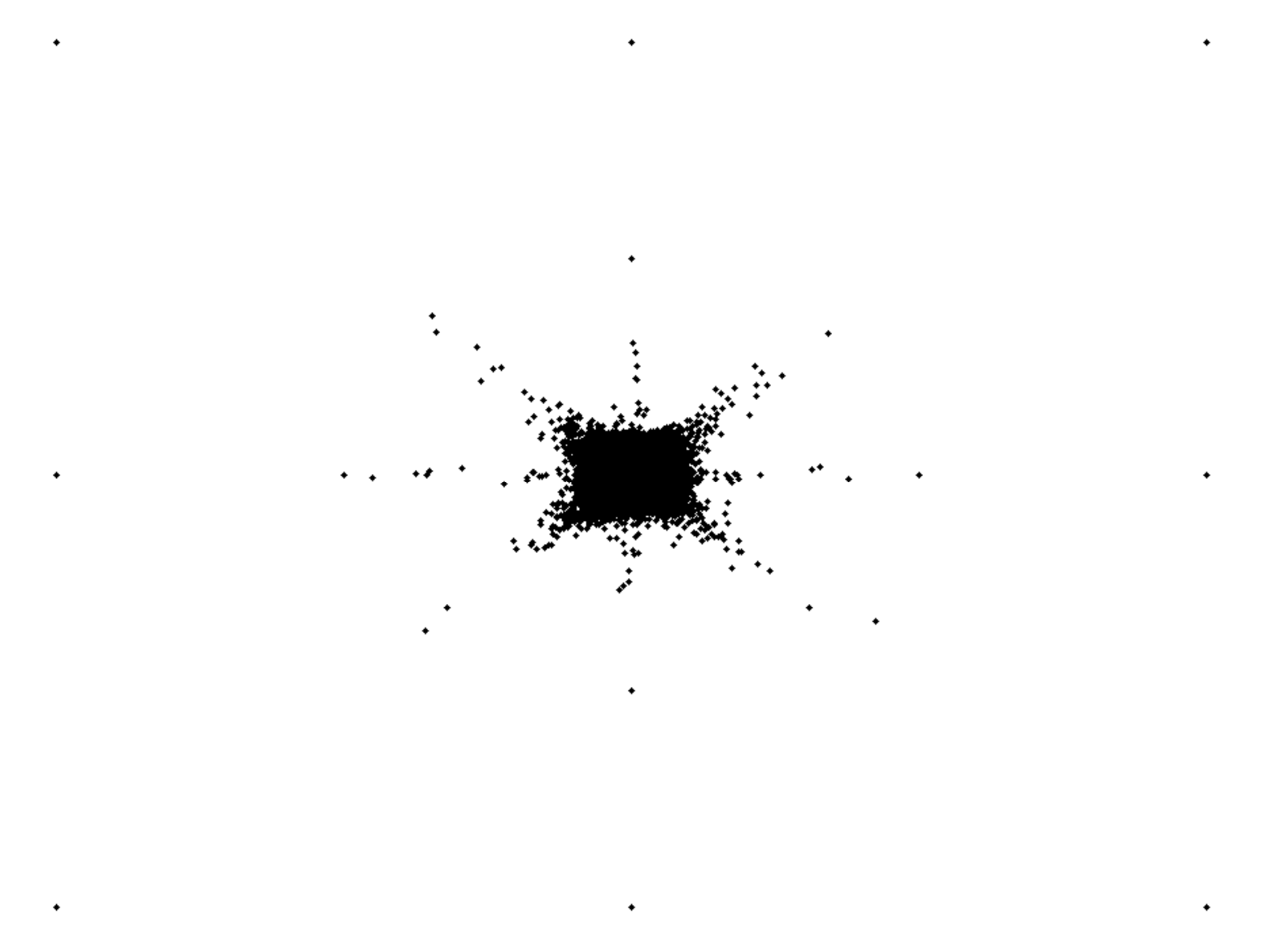};
          \draw[dashed, color=lightgray] 
            (axis cs:-1,-1) -- (axis cs:-1,1) -- 
            (axis cs: 1,1) -- (axis cs: 1,-1)-- 
            (axis cs:-1,-1);
        \end{axis}
      \end{scope}    
    \end{tikzpicture}
    %\label{Fig_dir_mthd_a1b1_2500}
  \end{minipage}
  \caption{Approximation of $\rs{F_{1,1}}$ by a direct approach, $80$ (left) and $2500$ (right) iterations, grid range $0.001$ each} 
  \captionsetup{{width=\textwidth}}
  \label{Fig_dir_mthd_a1b1}
\end{figure}

\noindent
It is effects like this which
prevent a direct approach from producing reasonable results
\cite{Guiheneuf2015RotationSetComputation}, and frequently lead to an
underestimation of the actual rotation set. In order to counter this, one would
have to choose the grid constant $N$ exponentially large with respect to $n$ (of
magnitude $N\sim L^n$, where $L$ is a Lipschitz constant of $f$), which is not
feasible for efficient computations.

Moreover, the above discussion shows that in the direct approach there is a
critical dependency between the parameters $N$ (grid size) and $n$ (number of
iterates), which cannot be chosen independent of each other. In contrast to
this, it is surely desirable to have an algorithm whose precision increases in
both parameters independently, so that computation capacities can be pushed to
their limit without having to worry about the precise relation of the
parameters. This is exactly what the set-oriented method will allow us to do.

\subsection{The set-oriented approach -- description of the algorithm.}
 \label{algorithm}
We start with a basic observation. By Lemma~\ref{main_Kn_convH}, the sets
$K_n(F)$ converge to the rotation set $\rs{F}$, and moreover we have the
elementary estimate
\begin{align}
  \dH\left(\frac{1}{n}F^n([0,1]^2),K_n(F) \right)\ \leq \ \frac{\sqrt{2}}{n}
  \label{E_Fn_convH2} \ .
\end{align}
Thus, we obtain
\begin{corollary} \label{Col_Fn_convH}
  Let $F \in \hm{2}$. Then
%  \begin{align*}
%    \frac{F^n([0,1]^2)}{n} \convH \rs{F} \text{ as } n \to \infty.
%  \end{align*}
  $\frac{1}{n}F^n([0,1]^2) \convH \rs{F}$ as $n \to \infty$.
\end{corollary}

\noindent
This allows to focus on the sets $\frac{1}{n}F^n([0,1]^2)$ in order to compute
the rotation sets, which is quite convenient from a practical viewpoint. Hence,
we aim for an accurate numerical approximation of the sets
$\frac{1}{n}F^n([0,1]^2)$ for some large $n \in \mathbb{N}$. Inspired by the
software package GAIO (Global Analysis of Invariant Objects), we apply the
concept of \textit{box coverings} to formulate our algorithm. This library,
developed by Dellnitz, Froyland and Junge
\cite{DellnitzFroylandJunge2001GAIO}, provides numerical methods for
the analysis of dynamical systems by set-oriented techniques. As mentioned
above, on the theoretical level this corresponds to considering
$\eps$-pseudo-orbits.

\bigskip
\noindent
For a given maximal diameter $\epsilon > 0$ let $\mathcal{B}_0$ be a collection
of compact sets $B \subset [0,1]^2$ with
\begin{align}
  \sup\limits_{B \in \mathcal{B}_0} 
  \diam(B)\ \leq \ \epsilon
  \quad \text{ and } \quad
  \bigcup\limits _{B \in \mathcal{B}_0} B \ = \ [0,1]^2 \ .
  \label{E_def_covI2}
\end{align}
Note that $\cB_0$ can be considered as a covering of $\T^2$, whose lift to
$\R^2$ is then given~by
\begin{align*}
  \mathcal{B} \ = \ \left\{ B + t \; \big| \; B \in \mathcal{B}_0,\ t \in
    \mathbb{Z}^2 \right\} \ .
\end{align*}
The elements of the collections $\mathcal{B}_0$ and $\cB$ will be referred to as
\textit{boxes}. Further, given $\eta>0$, for each box $B\in\cB$, let
$\Gamma_B\ssq B$ be a finite set of points which is $\eta$-dense in $B$ and
chosen such that $\Gamma_{B+t}=\Gamma_B+t$ for every integer vector
$t\in\Z^2$. Finally, fix $R>0$. Then, for $F \in \hm{2}$, we now generate a
sequence of sets $\left( Q_n^\ast\right)_{n\in\mathbb{N}}$ (approximations of
the rotation set) according to the following algorithm.

\bigskip
 %% 1. row %%%%%%%
\begin{minipage}[t]{0.18\textwidth} \vspace{0pt} \raggedright {\em Initialisation:} \end{minipage}
\begin{minipage}[t]{0.15\textwidth}  \vspace{0pt}  \raggedleft  $Q_0$ \end{minipage}
\begin{minipage}[t]{0.5\textwidth} \vspace{0pt} \raggedright
    $\ =\ [0,1]^2  $
\end{minipage} \hfill
\begin{minipage}[t]{0.1\textwidth} \, \end{minipage}

\medskip
 %% 2. row %%%%%%%
\begin{minipage}{0.18\textwidth} \raggedright {\em Box images:} \end{minipage}
\begin{minipage}{0.7\textwidth} \raggedright
 Given $B\in\cB$,  let the {\em box image}   of $B$ be defined as
\end{minipage}

 %% 3. row %%%%%%%
\begin{minipage}[t]{0.18\textwidth} \vspace{0pt} \quad \end{minipage}
\begin{minipage}[t]{0.15\textwidth}  \vspace{0pt}  \raggedleft  $\cI(B)$ \end{minipage}
\begin{minipage}[t]{0.5\textwidth} \vspace{0pt} \raggedright
    $\ =\ \{B'\in\cB \mid \exists x\in\Gamma_B: d(F(x),B')\leq R\} $
\end{minipage} \hfill
\begin{minipage}[t]{0.1\textwidth} \vspace{0pt} \raggedleft  \vspace{-10pt}
  \begin{align} \, \label{e.box_images} \end{align}
\end{minipage}

\medskip
 %% 4. row %%%%%%%
\begin{minipage}[t]{0.18\textwidth} \vspace{0pt} \quad \end{minipage}
\begin{minipage}{0.7\textwidth} \raggedright
   (see Fig. \ref{f.box_covering}).
\end{minipage}

\smallskip
 %% 5. row %%%%%%%
\begin{minipage}{0.18\textwidth} \raggedright {\em Iteration:} \end{minipage}
\begin{minipage}{0.7\textwidth} \raggedright
  For $k = 0, \ldots, n-1$ generate the  {\em box coverings}
\end{minipage}

 %% 6. row %%%%%%%
\begin{minipage}[t]{0.18\textwidth} \vspace{0pt} \quad \end{minipage}
\begin{minipage}[t]{0.15\textwidth}  \vspace{0pt}  \raggedleft  $\mathcal{B}_{k+1}$ \end{minipage}
\begin{minipage}[t]{0.45\textwidth} \vspace{0pt} \raggedright
    $\ = \ \left\{ B' \in \cB\mid \exists B \in \mathcal{B}_k: B' \in \cI(B) \right\}, $
\end{minipage} \hfill
\begin{minipage}[t]{0.1\textwidth} \vspace{0pt} \raggedleft  \vspace{-10.5pt}
  \begin{align} \, \label{E_box_coll_Bn} \end{align}
\end{minipage}

\smallskip
 %% 7. row %%%%%%%
\begin{minipage}[t]{0.18\textwidth} \vspace{0pt} \quad \end{minipage}
\begin{minipage}[t]{0.15\textwidth}  \vspace{0pt}  \raggedleft  $ Q_{k+1}$ \end{minipage}
\begin{minipage}[t]{0.45\textwidth} \vspace{0pt} \raggedright
    $\ = \ \bigcup\nolimits_{B \in \mathcal{B}_{k+1}} \!\!B. $
\end{minipage} \hfill
\begin{minipage}[t]{0.1\textwidth} \vspace{0pt} \raggedleft \vspace{-10.5pt}
  \begin{align} \, \label{E_def_Qn} \end{align}
\end{minipage}

\medskip
 %% 8. row %%%%%%%
\begin{minipage}{0.18\textwidth}   \raggedright {\em Normalisation:} \end{minipage}
\begin{minipage}{0.15\textwidth}   \raggedleft  $Q_n^\ast$ \end{minipage}
\begin{minipage}{0.45\textwidth} \raggedright
    $\ = \ \frac{1}{n}Q_{n} ,\  n \in \mathbb{N}$
\end{minipage} \hfill
\begin{minipage}{0.1\textwidth} \raggedleft \vspace{-10.5pt}
  \begin{align} \, \label{e.Qnstar} \end{align}
\end{minipage}

\bigskip
\begin{figure}[!ht]
\raggedright
\begin{minipage}[m]{0.51\textwidth}
   \raggedleft
   \begin{tikzpicture}[scale=0.7]
      \draw (0,0)--(0,3)--(3,3)--(3,0) -- cycle; 
       
      \foreach \i in {0,0.75,...,3} {
          \foreach \j in {0,0.75,...,3} {
         \draw[fill=lightgray,color=lightgray]  (\i,\j) circle (.4ex) ;}}
         
      % special point   
      \draw[fill=black,color=black] (2.25,0.75) circle (.4ex) ;
      
      % notations 
       \node at (-1,2.6) {$B \in \mathcal{B}_k$};      
       \node at (4.4,0.75) {$x \in \Gamma(B)$};    
       \node at (5.5,3.1) {$F$} ;
        
      \draw [decorate,decoration={brace,amplitude=2pt},xshift=0pt,yshift=0pt]
 (1.125,1.875) -- (1.5,2.25) node [black,above=1.5,left=1.5] {$\eta$};          
 
      \path [->,bend angle=25, bend left, line width = 1pt]  (4,2.4) edge (7,2.4); 
   \end{tikzpicture}
\end{minipage}
\begin{minipage}[m]{0.31\textwidth}
   \raggedright
   \begin{tikzpicture}[scale=1]
      % hit boxes
      \fill[fill=mygray, color=mygray] 
            (0,0) -- (0,4) -- (3,4) -- (3,3) -- (4,3) -- (4,1) -- (3,1) -- (3,0) -- cycle;  
      
      % dashed lines for box covering of R^2
      \foreach \i in {0,...,4} {
         \draw[dotted] (-0.3,\i) -- (4.3,\i);
         \draw[dotted] (\i,-0.3) -- (\i,4.3);}
         
      % image point
      \draw[fill=black]  (1.73,2.05) circle (.4ex) node[below = 2]{$F(x)$};
      
      % R-environment of image point
      \draw[radius = 1.41,style = dashed ] (1.73,2.05) circle ;
      
      % radius line
      \draw  (1.73,2.05) -- (2.73,3.05) node[below = 10, left = 11]{$R$};
      \draw (3,0) -- (4,1) node[below = 10, left = 11]{$\epsilon$};
      
      % notation of box image
       \node at (-0.5,3.6) {$\cI(B)$};
   \end{tikzpicture}
\end{minipage}
\caption{Box image of a box $B$ in the box covering $\mathcal{B}_k$, one test point, exemplarily.}
\label{f.box_covering}
\end{figure}
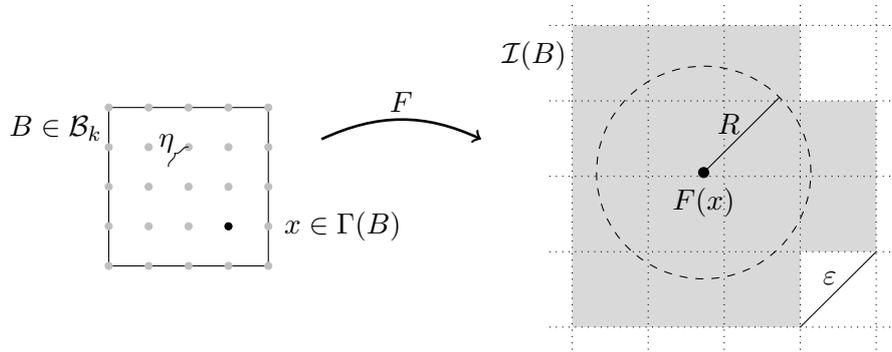

%\begin{enumerate}[1.]
%  \item (\textit{Initilization}) 
%    %\hspace{0.33cm} 
%    $Q_0  \coloneqq [0,1]^2$.
%   \item (\textit{Iteration})
%    For $k\in \{0,\ldots,n-1\}$ define
%    \begin{align}
%      \;\; \mathcal{B}_{k+1} & 
%       \coloneqq \left\{ \hat{B} \in \hat{\mathcal{B}} \mid
%                          \exists B \in \mathcal{B}_k: 
%                          F(B) \cap \hat{B} \neq \emptyset \right\} 
%                          \label{E_box_coll_Bn} \\
%           &\, = \left\{ \hat{B} \in \hat{\mathcal{B}} \mid
%                          F(Q_k) \cap \hat{B} \neq \emptyset \right\}
%                          \nonumber \\ %\label{E_box_coll_Qn} \\
%                        &  \nonumber \\ % just a free line
%      \;\; Q_{k+1} &\coloneqq \bigcup\limits_{B \in \mathcal{B}_{k+1}} B.
%      \label{E_def_Qn}
%    \end{align}
%  \item (\textit{Normalization})
%    %\begin{align*}
%     \hspace{1.06cm} $Q_n^\ast \coloneqq \frac{1}{n} Q_{n} $
%    %\end{align*}
%\end{enumerate}

\clearpage
\noindent
The starting point of our algorithm is the unit square $[0,1]^2$. Instead
of iterating single test points, we consider collections of boxes $B \in
\mathcal{B}$ and always collect those which are hit by an image of one of the
previous boxes. 

By choosing the parameter $R$ for the numerical calculation of
the box images sufficiently large, we can ensure that no boxes are missed out
and the images are always overestimated (see
Lemma~\ref{l.box_overestimation}). This leads to a systematic overestimation of
the sets $\frac{1}{n}F^n\left([0,1]^2\right)$, which will eventually lead to the
fact that the algorithm always yields a superset of the actual rotation set (up
to a small shift of order $\nicefrac{2\sqrt{2}}{n}$, see Lemma~\ref{L_dH_Kn_Kneps}).  Hence,
the effect of underestimation described for the direct approach in the previous
subsection can be excluded (see Theorem~\ref{Th_box_err_asympt} below). 

 Note
also that for the implementation of the algorithm, the box images only have to
be computed once of each box $B\in\cB_0$ at the beginning. This immediately
provides the box images for the respective integer translates as well, and the
information can then be used throughout the whole iterative procedure.

\subsection{Convergence of the approximations -- qualitative results and error
  bounds.}

Relations between the normalised approximations $Q_n^\ast$ and the rotation set
$\rs{F}$ are established by the following two
results. Theorem~\ref{Th_box_err_asympt} is qualitative in nature and ensures
convergence, whereas Theorem~\ref{Thm_main_alg_err} uses an additional
boundedness assumption to provide error estimates.

\begin{theorem} \label{Th_box_err_asympt} Suppose that $F \in \hm{2}$ is
  Lipschitz continuous with Lipschitz constant $L> 1$ and for each $\eps>0$
  the constants $\eta,R>0$ are chosen such that $L\eta\leq R\leq \eps$.  Then
  \begin{equation}
    \label{eq:Qn-convergence}
    \lim_{\substack{n\to\infty\\\eps\to 0}} \ \dH(Q_n^\ast,\rs{F}) \ = \ 0 \ ,
  \end{equation}
  in the sense that for all $\delta>0$ there exist $\eps_0>0$ and $n_0\in\N$
  such that if $\eps$ in (\ref{E_def_covI2}) satisfies $\eps \in(0,\eps_0]$ and
  $n\geq n_0$, then $\dH(Q_n^\ast,\rs{F}) <  \delta$.
\end{theorem}

\begin{rem}
  For theoretical purposes, one may also want to ignore the fact that the images
  of the boxes $B\in\cB$ can only be approximated via test points, and define
  alternative sequences $\big{(} \tilde Q_n \big{)}_{n \in \mathbb{N}},\big{(} \tilde Q_n^* \big{)}_{n \in \mathbb{N}}$ by using the precise images via
  \begin{align*}
    \tilde Q_0 &\ = \ Q_0 \ ,\\
% \tilde Q_n &\ =\ \bigcup\limits_{\substack{B\in \cB,\\B\cap F(Q_n)\neq \emptyset}}B\ ,\\
   \tilde Q_n &\ =\ \bigcup \left\{ B\in \cB \mid  B\cap F(\tilde Q_{n-1})\neq \emptyset \right\} \ ,\\
    \tilde Q_n^* &\ =\ \frac{1}{n} \tilde Q_n \ .
  \end{align*}
  This corresponds to setting the parameters $\eta$ and $R$ to zero. Then the
  above convergence result is still valid for the new sequence, that is,
  $\lim_{\substack{n\to\infty\\\eps\to 0}} \ \dH(\tilde
  Q_n^\ast,\rs{F})=0$. Moreover, the assumption of Lipschitz continuity is not
  needed in this case, but it would still be required in the respective analogue
  of the error estimates given in Theorem~\ref{Thm_main_alg_err} below.
\end{rem}

\noindent
For the proof of Theorem~\ref{Th_box_err_asympt}, we will need the following
lemma.

\begin{lemma}
  \label{l.box_overestimation} Suppose that $F \in \hm{2}$ is
  Lipschitz continuous with Lipschitz constant $L>1$ and for each $\eps>0$ the
  constants $\eta,R>0$ are chosen such that $L\eta\leq R\leq \eps$. Then
%  \begin{eqnarray*}
%  \lefteqn{\frac{F^n\left([0,1]^2\right)}{n}  \ \ssq \ Q^*_n} \\ & \ssq & \left\{\xi_n/n \mid
%  (\xi_j)_{j=0}^n \textrm{ is a } 2\eps\textrm{-pseudo-orbit of } F \textrm{
%    with } \xi_0\in[0,1]^2\right\} \ .
%  \end{eqnarray*}
    \begin{align*}
  \frac{F^n\left([0,1]^2\right)}{n}  \ \ssq \ Q^*_n 
  \ \ssq \  \left\{\frac{\xi_n}{n}\; \bigg| \;
  (\xi_j)_{j=0}^n \textrm{ is a } 2\eps\textrm{-pseudo-orbit of } F \textrm{
    with } \xi_0\in[0,1]^2\right\} \ .
  \end{align*}
\end{lemma}
\begin{proof}
  First, suppose that $v=\frac{1}{n} F^n(x_0)\in \frac{1}{n}F^n\left([0,1]^2\right)$ and let
  $x_j=F^j(x_0)$. Choose boxes $B_j\in\cB$ such that $x_j\in B_j$. We claim that
  $B_j\ssq Q_j$ for $j=0\ld n$, so that finally $x_n=nv\in B_n\ssq Q_n$ and
  hence $v\in Q_n^*$.

  For $j=0$ the claim is obvious. Therefore, suppose that $B_j\ssq Q_j$. Then
  there exists a test point $y\in\Gamma_{B_j}$ $\eta$-close to $x_j$, so that
  $d(F(y),x_{j+1})\leq L\eta\leq R$. Since $x_{j+1}\in B_{j+1}$, this implies
  $B_{j+1}\in \cI(B)$ and thus $B_{j+1}\in Q_{j+1}$.

  This shows the first inclusion $\frac{1}{n} F^n\left([0,1]^2\right)\ssq Q^*_n$. In order
  to show the second inclusion, suppose that $v\in Q^*_n$. Then, by definition
  of $Q_n$, there exists a sequence of boxes $B_0\ld B_{n}$ such that $nv\in B_n$
  and $B_{j+1}\in\cI(B_j)$ for all $j=0\ld n-1$.  By definition of the box
  images, there exists a sequence of test points $\xi_j\in \Gamma_{B_j}$, $j=0\ld n-1$,
  such that $d(F(\xi_j),B_{j+1})\leq R$. Hence
  \[
  d(F(\xi_j),\xi_{j+1}) \ \leq \ R+\eps \ \leq 2\eps \ .
  \]
  for all $j=0\ld n-2$, and this remains true for $j=n-1$ if we let $\xi_n=nv\in
  B_n$. This means that $(\xi_j)_{j=0}^n$ is a $2\eps$-pseudo-orbit and thus
  yields the required second inclusion. 
\end{proof}

\begin{corollary} \label{c.box-overestimation}
  In the situation of Lemma~\ref{l.box_overestimation}, we have 
  \[
   K_n(F) \ \ssq \
  B_{\frac{\sqrt{2}}{n}}\left(Q_n^*\right) \eqand 
  Q_n^* \ \ssq \ B_{\frac{\sqrt{2}}{n}}\left(K_n^{2\eps}(F)\right) \ . 
  \]
\end{corollary}

\begin{proof}[\textbf{\textit{Proof of Theorem~\ref{Th_box_err_asympt}}}]
  First, by Theorem~\ref{main} and the convexity of the rotation set, there
  exists $\eps_0>0$ such that
\[
\dH\left(\conv\left(\epsrs{F}{2\eps_0}\right),\rs{F}\right) \ < \ \frac{\delta}{3} \ .
\]
Further, by the definition of $\epsrs{F}{2\eps_0}$, there exists $n_0\in\N$ such
that for all $n\geq n_0$ we have $K_n^{2\eps_0}(F) \ssq
B_{\nicefrac{\delta}{3}}\left(\epsrs{F}{2\eps_0}\right)$ and hence
\[
\conv\left(K_n^{2\eps_0}(F)\right) \ \ssq
\ B_{\frac{\delta}{3}}\left(\conv\left(\epsrs{F}{2\eps_0}\right)\right) \ .
\]
As for all $n\in\N$ we have $\rs{F}\ssq
\conv\left(K_n(F)\right)\ssq\conv\left(K_n^{2\eps_0}(F)\right)$ (see
Lemma~\ref{LProper3}), we obtain that
\begin{equation}\label{e.qualitative_proof_1}
\dH\left(\conv\left(K^{2\eps_0}_n(F)\right),\rs{F}\right) \ < \frac{2\delta}{3} \ .
\end{equation}
At the same time, Lemma~\ref{main_Kn_convH} implies that we can choose $n_0$ such
that
\begin{equation}\label{e.qualitative_proof_2}
  \dH\left(K_n(F),\rs{F}\right) \ < \ \frac{2\delta}{3}
\end{equation}
for all $n\geq n_0$. If $n_0$ is chosen such that $\nicefrac{\sqrt{2}}{n_0} < \nicefrac{\delta}{3}$ and $n\geq n_0$, then
using (\ref{e.qualitative_proof_2}) together with the first inclusion in
Corollary~\ref{c.box-overestimation}, we obtain
\[
\rs{F} \ \ssq \ B_{\frac{2\delta}{3}}(K_n(F)) \ \ssq \ B_\delta(Q^*_n) \ .
\]
Conversely, (\ref{e.qualitative_proof_1}) together with the second inclusion in
Corollary~\ref{c.box-overestimation} yield
\[
Q^*_n \ \ssq \ B_{\frac{\delta}{3}}\! \left(K_n^{2\eps_0}(F)\right) \ \ssq \ B_\delta(\rs{F}) \ .
\]
Together with the fact that $K_n^{2\eps}(F)\ssq K_n^{2\eps_0}(F)$ for all
$\eps<\eps_0$, this means that
\[
\dH(Q^\ast_n,\rs{F}) \ < \ \delta 
\]
whenever $n\geq n_0$ and $\eps\in(0,\eps_0]$, as required.
\end{proof}

\noindent
In general, it is not possible to give quantitative error estimates for the
numerical computation of rotation sets. The reason is that there are no general
apriori bounds for the convergence of the sets $K_n(F)$ to $\rs{F}$. However,
it turns out that in many situations, and in particular whenever the rotation
set has non-empty interior, there exists a positive constant $c>0$ such that
\begin{align}
  \dH (K_n(F),\rs{F}) \leq \frac{c}{n} \quad \textrm{ for all } n\in\N\ .
  \label{E_err_plausible}
  \tag{BD}
\end{align}
This fact has been proven for diffeomorphisms in
\cite{AddasZanata2015BoundedMeanMotionDiffeos} and the result was later
generalised to homeomorphisms in \cite{LeCalvezTal2015ForcingTheory}. This is
also referred to as {\em bounded deviation property}. In our context,
(\ref{E_err_plausible}) together with the existence of a Lipschitz constant
allows to provide the following quantitative estimates.

\begin{theorem} \label{Thm_main_alg_err} Suppose $F \in \hm{2}$ is Lipschitz continuous with Lipschitz constant $L> 1$, let $\epsilon > 0$, $\eta L\leq R\leq \eps$ and $n \in \mathbb{N}$. Further, assume that (\ref{E_err_plausible}) holds for
  $c>0$. Then
\begin{align}
  \dH \left(Q_n^\ast,\rs{F}\right)\  \leq \ 
  \max \left\{ \frac{2\sqrt{2}}{n}, 
               \frac{\sqrt{2}}{n} + \gamma_{\epsilon,n} \right\}  \ , 
  \label{E_main_err}
\end{align}
where 
\begin{align}
  \gamma_{\epsilon,n} \ = \ 
%      \frac{r_n}{n}(M+\epsilon) +
%      \Big{(} 1- \frac{r_n}{n}\Big{)}
%      \min\limits_{1\leq k\leq n} 
%      \frac{1}{k} \left( c + \epsilon \frac{L^k-1}{L-1} \right)
  \frac{2r_n}{n}(M+\epsilon) + \left( 1- \frac{r_n}{n} \right)
  \min\limits_{1\leq k\leq n} \frac{1}{k} \left( c + 2\epsilon \frac{L^k-1}{L-1}
  \right) \ .
  \label{E_min_gamma} 
\end{align} 
Here $M =\max_{x\in [0,1]^2} \|F(x)-x\|$, $k_n$ is the number between $1$ and
$n$ for which the minimum on the right is attained and $r_n = n \bmod k_n$.
\end{theorem}

\noindent
The proof is given in
Section~\ref{QuantitativeEstimates} below. 

\begin{rem}  \label{Rem_err_nstar}
  \begin{itemize}
  \item[(a)] It should be noted that the above estimate is rather of theoretical
    than practical interest. This is exemplified in part (b) of this remark
    below. We include it nevertheless, since on the one hand it demonstrates
    what is possible on the analytic side, and on the other hand the proof
    reflects and demonstrates very well how and why the nonlinearity of the
    dynamics complicates the efficient computation of rotation sets. 
  \item[(b)] In order to see why the above estimates are hardly relevant for the
    numerical implementation, suppose that the constant $c$ in
    (\ref{E_err_plausible}) is known (which is usually not the case) and
    relatively small, say, equal to $1$. Even in this case, in order to achieve
    an apriori error bound of order $10^{-2}$, the integer $k$ in the term
    $\frac{1}{k} \big{(} c + 2\epsilon \frac{L^k-1}{L-1} \big{)}$ in
    (\ref{E_min_gamma}) would have to be at least $100$ -- otherwise $\nicefrac{c}{k}> 100$
    -- but then at the same time $\eps$ would need to be smaller than
    $\big{(}\frac{2(L^k-1)}{L-1}\big{)}^{{ }_{-1}}$. Hence, even if $L$ is only $2$,
    this would require to work with a box diameter of $\eps\simeq 2^{-100}$,
    which is hardly possible with contemporary computers.
  \item[(c)] Fortunately, it turns out that in the actual implementation the
    convergence is usually much faster than indicated by the above error
    estimates. One possible reason is the fact that $f$ may possess a shadowing
    property. This leads to improved error bounds, where essentially the
    exponential term $\eps\frac{L^k-1}{L-1}$ can be dropped altogether. We
    discuss this in detail in the following subsection.
  \end{itemize}
\end{rem}

\subsection{Implications of shadowing.}\label{Shadowing}

Given a self map $g:X\to X$ of some metric space $X$, we say an orbit
$\nfolge{x_n}$ of $g$ {\em $\delta$-shadows} a sequence $\nfolge{\xi_n}$ if
$d(x_n,\xi_n)<\delta$ for all $n\in\N$. Given $\delta,\eps>0$, we say $g$ has
the {\em $\delta$-shadowing property with constant $\eps$} if all
$\eps$-pseudo-orbits of $g$ are $\delta$-shadowed by some orbit of $g$. If such
a constant $\eps=\eps(\delta)$ exists for all $\delta>0$, we simply say that $g$
has the {\em shadowing property}.

\clearpage
\begin{theorem} \label{t.shadowing} Suppose $f$ is a torus homeomorphism
  homotopic to the identity with lift $F\in\hm{2}$ and $\delta,\gamma\in (0,\nicefrac{1}{2})$
  are such that $d(x,y)<\delta$ implies $d(f(x),f(y))<\nicefrac{1}{2}-\gamma$ for all
  $x,y\in\T^2$. Further, assume that $f$ has the $\delta$-shadowing property
  with constant $\eps\in(0,\gamma)$. Then 
  \[
    \epsrs{F}{\eps} \ = \ \rs{F} \ .
  \]
\end{theorem}
\begin{corollary}
  If a torus homeomorphism $f$ homotopic to the identity has the shadowing
  property, then there exists $\eps_0>0$ such that $\epsrs{F}{\eps}=\rs{F}$ for
  any $\eps\in(0,\eps_0]$ and any lift $F$ of $f$.
\end{corollary}

\noindent
For the proof of Theorem~\ref{t.shadowing}, we need the following elementary
statement.
\begin{lemma}
  Suppose $f,F,\delta,\gamma,\eps$ are chosen as in
  Theorem~\ref{t.shadowing}. Then $F$ has the $\delta$-shadowing property with
  constant $\eps$.
\end{lemma}
\begin{proof}
  Suppose $\nfolge{\hat\xi_n}$ is an $\eps$-pseudo-orbit for $F$. Then
  $\xi_n=\pi(\hat\xi_n)$ defines an $\eps$-pseudo-orbit of $f$, and we can
  therefore find some $x_0\in\T^2$ such that $d\left(f^n(x_0),\xi_n\right)<\delta$ for all
  $n\geq 0$. Let $\hat x_0$ be the unique lift of $x_0$ in $B_\delta\big{(}\hat
  \xi_0\big{)}$. We claim that $d(F^n(\hat x_0),\hat\xi_n)<\delta$ for all $n\geq 0$, so that the orbit of $\hat x_0$ is the required $\delta$-shadowing~orbit.
 
  For $n=0$ there is nothing to prove. If $d(F^n(\hat x_0),\hat\xi_n)<\delta$
  for some $n\geq 0$, then $F^{n+1}(\hat x_0)\in
  B_{1/2-\gamma}\big{(}F(\hat\xi_n)\big{)}\ssq
  B_{1/2}\big{(}\hat\xi_{n+1}\big{)}$. However, as
  $B_{1/2}\big{(}\hat\xi_{n+1}\big{)}$ projects injectively to $\T^2$ and
  $d(\pi(F^{n+1}(\hat x_0),\pi(\hat\xi_{n+1})) =
  d(f^{n+1}(x_0),\xi_{n+1})<\delta$, we also obtain
  $d(F^{n+1}(\hat x_0),\hat\xi_{n+1})) <\delta$ as required.
\end{proof}
\begin{proof}[\textbf{\textit{Proof of Theorem~\ref{t.shadowing}}}]
  Under the assumptions of the theorem, any finite $\eps$-pseudo-orbit $(\hat
  \xi_j)_{j=0}^n$ of $F$ is $\delta$-shadowed by some $F$-orbit $(x_j)_{j=0}^n$, so
  that
  \[
  \left| \frac{\hat\xi_n-\hat\xi_0}{n} - \frac{F^n(x_0)-x_0}{n}\right| \ \leq \
  \frac{2\delta}{n} \ .
   \]
   Therefore $\dH(K_n(F),K^\eps_n(F)) \leq \nicefrac{2\delta}{n}$, and due to the definition
   of the sets $K_n(F)$, $K^\eps_n(F)$ and $\epsrs{F}{\eps}$ and the convergence
   $K_n(F) \convH \rs{F}$ as $n \to \infty$ by Lemma~\ref{main_Kn_convH}, this implies
   $\epsrs{F}{\eps}=\rs{F}$.
\end{proof}

\noindent
As mentioned before, the shadowing property leads to improved error estimates
for the set-oriented computation of the rotation set, and in particular allows
to eliminate the exponential term $\gamma_{\eps,n}$ in
Theorem~\ref{Thm_main_alg_err}.
\begin{theorem}\label{t.shadowing_error}
  Suppose that $f$ satisfies the assumptions of
  Theorem~\ref{Thm_main_alg_err} and the bounded deviations hypothesis
  (\ref{E_err_plausible}) with constant $c>0$. Let $F\in\hm{2}$ be a lift of $f$
  and $Q_n^\ast$ be defined by (\ref{E_def_Qn}). Then
 \[
        \dH(Q^\ast_n,\rs{F}) \ \leq \ \frac{\sqrt{2}+1+c}{n} \ .
 \]
\end{theorem}
\noindent
The proof is postponed until the end of the next subsection.

\subsection{Quantitative estimates -- proofs of Theorems~\ref{Thm_main_alg_err}
  and \ref{t.shadowing_error}.}
\label{QuantitativeEstimates}

Throughout this section, we assume that $F\in\hm{2}$  and $Q_n,\ Q_n^\ast$
are defined as in the preceding section. Let
%satisfies the bounded deviations hypothesis (\ref{E_err_plausible}) with $c>0$
\begin{align*}
  P_n^\epsilon(F) \ = \  \left\{ \xi_n \in \mathbb{R}^2 \mid (\xi_j)_{j=0}^n
    \text{ is an } \epsilon\text{-pseudo-orbit of } F \text{ with } \xi_0 \in
    [0,1]^2 \right\}
\end{align*}
and note that $P^0_n(F)=F^n\left([0,1]^2\right)$. Then the statement of
Lemma~\ref{l.box_overestimation} can be rewritten as 
\begin{align}
  P^0_n(F) \ \ssq \ Q_n \ \subseteq \ P_n^{2\epsilon}(F) \ .
  \label{E_def_Pneps}
\end{align}
This leads to the following initial
estimate.

\begin{lemma} \label{Prop_mainalg_1} Suppose that $F \in \hm{2}$ is
  Lipschitz continuous with Lipschitz constant $L>1$ and for each $\eps>0$ the
  constants $\eta,R>0$, are chosen such that $L\eta\leq R\leq \eps$. Then, for
  every $n\in\mathbb{N}$, we have
  \begin{displaymath}
    \rs{F} \ \subseteq \ \overline{B_{\frac{2\sqrt{2}}{n}}(Q_n^\ast)} \ .
  \end{displaymath}
\end{lemma}
 \vspace{-8pt}
\begin{proof}
  We have 
  \begin{eqnarray*}
    \rs{F} & \stackrel{\textrm{Lemma~\ref{LProper3}}}{\ssq} & \conv(K_n(F)) \ \ssq 
    \ \overline{B_{\frac{\sqrt{2}}{n}}\left(\frac{1}{n}\conv\big{(}F^n([0,1]^2\big{)}\right)} \\ & \stackrel{\textrm{Lemma~\ref{l.MZ-estimate}}}{\ssq} &
    \overline{B_{\frac{2\sqrt{2}}{n}}\left(\frac{1}{n}\big{(}F^n([0,1]^2\big{)}\right)} \ \stackrel{\textrm{Lemma}~\ref{l.box_overestimation}}{\ssq} \ 
    \overline{B_{\frac{2\sqrt{2}}{n}}(Q_n^\ast)} \ . 
  \end{eqnarray*}
\end{proof}

\begin{lemma} \label{L_dH_Kn_Kneps} Suppose that $F \in \hm{2}$ is
  Lipschitz continuous with Lipschitz constant $L>1$ and $\eps>0$. Then, for
  all $n\in\N$, we have
  \begin{align*}
    \dH \big{(}K_n^{2\epsilon}(F), K_n(F)\big{)} \ \leq \ \frac{2\epsilon(L^n-1)}{n(L-1)} \
    \eqqcolon \kappa_{\epsilon,n}\ .
  \end{align*}
\end{lemma} 

\begin{proof}
  %The inclusion $K_n(F) \subseteq K_n^\epsilon(F)$ is evident. 
  Let $(\hat\xi_j)_{j=0}^ {n}$ be an $2\epsilon$-pseudo-orbit of $F$ with
  $\hat\xi_0 \in [0,1]^2$. Using the fact that $\|\hat\xi_1 - F(\hat\xi_0)\|
  \leq 2\epsilon$ and
  \begin{align*}
    \|\hat\xi_n-F^j(\hat\xi_0)\| %= \|F(\xi_{n-1})+y_{n-1}-F^n(\xi_0)\| 
   \ \leq \ \|F(\hat\xi_{j-1}) - F(F^{j-1}(\hat\xi_0))\| + 2\epsilon
    \ \leq \ L  \|\hat\xi_{j-1}-F^{j-1}(\hat\xi_0)\| + 2\epsilon,
  \end{align*}
  for all $j=1\ld n-1$, we recursively obtain the estimate
  \begin{align*}
    \| \hat\xi_n - F^n(\hat\xi_0)\| \ \leq \ 2\epsilon \sum\limits_{i=0}^ {n-1} L^i \ = \
    2\epsilon \frac{L^n -1}{L-1} \ .
    %\label{E_Kneps_err_1}
  \end{align*}
  Thus, for any $v=\frac{1}{n}(\xi_n - \xi_0) \in
  K_n^{2\epsilon}(F)$ %by \reff{E_Kneps_err_1}
  we have 
  \begin{align*}
    \frac{\xi_n-\xi_0}{n} \ = \ \frac{\xi_n - F^n(\xi_0)}{n} +
    \underbrace{\frac{F^n(\xi_0)-\xi_0}{n}}_{\in K_n(F)} \ ,  
  \end{align*}
  so that   \vspace{-8pt}
  \[
  K^{2\eps}_n(F) \ \ssq \ \overline{B_{\frac{2\epsilon(L^n-1)}{n(L-1)}}(K_n(F))} \ .
  \]
  Since conversely we always have $K_n(F)\ssq K^{2\eps}_n(F)$, this proves the
  statement. 
\end{proof}

\begin{lemma} \label{Prop_Kneps_incl_rs} Suppose that $F \in \hm{2}$ is
  Lipschitz continuous with Lipschitz constant $L>1$ and $\eps>0$. Further,
  assume that $F$ additionally satisfies (\ref{E_err_plausible}) with
  $c>0$. Then for all $n\geq 0$ we have
  \begin{align*}
    K_n^{2\epsilon}(F) \  \subseteq \ \overline{B_{\gamma_{\eps,n}}(\rs{F})} \ ,
  \end{align*}
  where $\gamma_{\epsilon,n}$ is defined as in (\ref{E_min_gamma}).
\end{lemma}

\begin{proof}
  Let $k\in\mathbb{N}$. Applying the estimate for the Hausdorff distance between
  the sets~$K_k^{2\epsilon}(F)$ and $K_k(F)$, denoted by $\kappa_{\epsilon,k}$ in
  Lemma \ref{L_dH_Kn_Kneps}, and the assumption (\ref{E_err_plausible}), we
  obtain
\begin{equation}\begin{split}
 \lefteqn{ \conv\left(K_k^{2\epsilon}(F)\right)
   \  \subseteq \ \conv \left(B_{\kappa_{\epsilon,k}}(K_k(F)) \right) }
    \\
   &\subseteq \ \conv \left(B_{\frac{c}{k}+\kappa_{\epsilon,k}}(\rs{F}) \right) 
    \label{E_Kneps_incl1} \
   = \  \overline{B_{\frac{c}{k}+\kappa_{\epsilon,k}}(\rs{F})}\ .    
 \end{split}
\end{equation}
For the last equality, note that $\rs{F}$ is convex.
Let $k_n \in \{1,\ldots,n\}$ be the natural number for which the minimum in the
definition of~$\gamma_{\epsilon,n}$ in~(\ref{E_min_gamma}) is attained. Further, let $m_n\in\N,\ r_n \in
\{0,\ldots,k_n-1\}$ be such that $n = m_n k_n + r_n$.  By the inclusion
(\ref{diffproof1}) in the proof of Lemma~\ref{difficultest}, since $n \geq k_n$,
we know that
  \begin{align}
    K_n^{2\eps}(F) &\subseteq \ \overline{B_{\frac{r_n}{n}(M +
        2\epsilon)}\left(\left( 1 - \frac{r_n}{n} \right) \conv \left(
        K_{k_n}^{2\epsilon}(F) \right)\right)} \ .
    \label{E_Kneps_incl2}
  \end{align}
  Combining (\ref{E_Kneps_incl1}) and (\ref{E_Kneps_incl2}) and by the choice of $k_n$ and the definition (\ref{E_min_gamma})
  of~$\gamma_{\epsilon,n}$, we obtain
  \begin{eqnarray*}
    K_n^{2\eps}(F) & \subseteq &
    \overline{B_{\frac{r_n}{n}(M+2\epsilon)}\left(\left(1-\frac{r_n}{n}\right)
        B_{\frac{c}{k}+\kappa_{\eps,k_n}}(\rs{F})\right)} \\
    & = & \overline{B_{\frac{r_n}{n}(M+2\epsilon)+ 
        \left(1-\frac{r_n}{n}\right)\left(\frac{c}{k_n}+\kappa_{\epsilon,k_n} \right)}
      \left(\left(1-\frac{r_n}{n}\right)\rs{F}\right)} \\
    &\subseteq & \overline{B_{\frac{r_n}{n}(2M+2\epsilon) + \left(1-\frac{r_n}{n}\right) 
        \left(\frac{c}{k_n}+\kappa_{\epsilon,k_n}\right)}(\rs{F})} \\
    & = &\overline{B_{\gamma_{\epsilon,n}}(\rs{F})} \ .
  \end{eqnarray*}
  For the inclusion from the second to the third line,
  note that $\rs{F}\ssq \overline{B_M(0)}$, so that $\left(1-\frac{r_n}{n}\right)\rs{F}\ssq
  \overline{B_{\frac{r_n}{n}M}\left(\rs{F}\right)}$.
\end{proof}

\begin{proof}[\textbf{\textit{Proof of Theorem \ref{Thm_main_alg_err}}}.]

In Lemma \ref{Prop_mainalg_1} we deduced that
\begin{align*}
  \rs{F} \subseteq \overline{B_{\frac{2\sqrt{2}}{n}}(Q_n^\ast)} \ .
\end{align*}
Conversely, from (\ref{E_def_Pneps}) and Lemma~\ref{Prop_Kneps_incl_rs} we obtain
\begin{align*}
  Q_n^\ast \ =\ \frac{1}{n} Q_n
  \ \subseteq\  \frac{1}{n} P_n^{2\epsilon}(F)
  \ \subseteq\ \overline{B_{\frac{\sqrt{2}}{n}}\left(K_n^{2\epsilon}(F)\right)}
  \ \subseteq \ \overline{B_{\frac{\sqrt{2}}{n}+\gamma_{\epsilon,n}}(\rs{F})} \ .
  %\label{E_Qnstar_contained}
\end{align*} 
Altogether, we obtain the error estimate (\ref{E_main_err}). \
\end{proof}

\begin{proof}[\textbf{\textit{Proof of Theorem~\ref{t.shadowing_error}.}}]
  On the one hand, we have
  \begin{equation}
    \rs{F} \ \ssq \ B_{\frac{2\sqrt{2}}{n}}(Q^\ast_n) \
    \end{equation}
    by Lemma~\ref{Prop_mainalg_1}. On the other hand, we have shown in the proof of
    Theorem~\ref{t.shadowing} that $K_n^{2\eps}(F) \ssq B_{\nicefrac{2\delta}{n}}(K_n(F))
    \ssq B_{\nicefrac{1}{n}}(K_n(F))$ (note that $\delta<\nicefrac{1}{2}$ by assumption) and thus~obtain
  \begin{eqnarray*}
    Q^*_n & \ssq & \frac{1}{n} P^{2\eps}_n(F) \ \ssq \ B_{\frac{\sqrt{2}}{n}}(K^{2\eps_n}(F)) \\ &
    \ssq & B_{\frac{\sqrt{2}+1}{n}}(K_n(F)) \ssq B_{\frac{\sqrt{2}+1+c}{n}}(\rs{F}) \ .
  \end{eqnarray*}
This shows the required estimate.
\end{proof}

\section{Numerical implementation and results} \label{sec_nums}

\noindent
In order to implement the above algorithm and to apply it to some specific
examples, we consider a standard family of (lifts of) torus diffeomorphisms
given by
\begin{equation} \label{e.standard_family}
  F_{\alpha,\beta}:\mathbb{R}^2 \to\mathbb{R}^2 \ , \quad (x,y) \mapsto \big{(}x+\alpha\sin(2\pi(y+\beta\sin(2\pi x))),\ y+\beta\sin(2\pi x)\big{)}\ ,
\end{equation}
where $\alpha,\beta\in\R$. Note that $F_{\alpha,\beta}$ is obtained as the
composition of two skew shifts, $F_{\alpha,\beta}=F_\alpha\circ F_\beta$, where
\begin{align*}
    F_\alpha(x,y) = \big{(}x+\alpha\sin(2\pi y),\ y\big{)}  \eqand  
    F_\beta(x,y)  = \big{(}x,\ y+\beta\sin(2\pi x)\big{)}\ .
\end{align*}
See
\cite{LeboeufKurchanFeingoldArovas1990PhaseSpaceLocalization,Jaeger2011EllipticStars}
for previous numerical studies and \cite{MisiurewiczZiemian1989RotationSets} for
structurally similar examples. For specific parameter values, the rotation set of $F_{\alpha,\beta}$ can easily
be determined analytically, which allows to test the numerical algorithm in a
controlled setting.

\begin{lemma} \label{l.rotset_F11}
  $\rs{F_{1,1}}=[-1,1]^2. $
\end{lemma}
\begin{proof} This follows directly from two elementary observations.  First, we have the general estimate
$\rs{F_{\alpha,\beta}}\ssq [-\alpha,\alpha]\times [-\beta,\beta]$, as $\alpha$
  and $\beta$ are the maximal step sizes in the horizontal and vertical
  direction. Hence, $\rs{F_{1,1}} \ssq [-1,1]^2$.

Conversely, it is easily checked that the rotation vectors $(1,1),(-1,1),(1,-1)$
and $(-1,-1)$ are realised by the fixed points $(\nicefrac{1}{4},\nicefrac{1}{4}),(\nicefrac{3}{4},\nicefrac{1}{4}),(\nicefrac{1}{4},\nicefrac{3}{4})$
and $(\nicefrac{3}{4},\nicefrac{3}{4})$. By convexity, this means that
$[-1,1]^2\ssq\rs{F_{1,1}}$. 
\end{proof}

\noindent
For the initialisation of the algorithm in Section~\ref{algorithm} , we choose
$\cB_0$ to be the standard covering of $[0,1]^2$ by $k^2$ squares of side length
$\nicefrac{1}{k}$, $k \in \mathbb{N}$. Note that we can thus choose $\eps=\nicefrac{\sqrt{2}}{k}$ in
(\ref{E_def_covI2}). We fix a Lipschitz constant $L$ of $F_{\alpha,\beta}$ (for
example, $L=1+4\pi^2$ works for all $(\alpha,\beta)\in[0,1]^2$) and set $R=\eps$
in (\ref{e.box_images}). Moreover, for each $B\in\cB_0$ we choose $\Gamma_B$ as
a standard grid of $m^2$ points in $B$, so that $\Gamma_B$ is
$\nicefrac{\sqrt{2}}{(k(m-1))}$-dense in $B$ (see Fig. \ref{f.box_covering}). Thereby, we
choose $m=m(k)$ such that $\eta=\nicefrac{\sqrt{2}}{(k(m-1))}< \nicefrac{\eps}{L}$. 

\bigskip
\noindent
In order to keep the dependence on $k$ explicit, we will from now on write
$Q_{k,n}^*$, instead of $Q_n^*$, for the approximation defined in
(\ref{e.Qnstar}). Then the assumptions of Theorem~\ref{Thm_main_alg_err} with
$\eps=\nicefrac{\sqrt{2}}{k}$ are satisfied and we obtain that $\lim_{k,n\to\infty}
Q_{k,n}^*=\rs{F_{\alpha,\beta}}$, with error bound provided by
(\ref{E_main_err}) (and by Theorem~\ref{t.shadowing_error} if $F_{\alpha,\beta}$
has the shadowing property).  Figure~\ref{f.iterates_for_F11_1} shows
$Q_{k,n}^*$ for $F_{1,1}$ for $k=8$ and different values of $n$.

\begin{figure}[!hb]
\rule{0pt}{10pt}
\centering
\labelsize
\setlength\figureheight{5.75cm} 
\setlength\figurewidth{5.75cm}
\setlength{\ticklength}{1mm}
\begin{center}
  \begin{minipage}[t]{0.45\textwidth}
    \vspace{0pt} % causes minipages being top-aligned
    \centering
\begin{tikzpicture} % n=1
      \begin{scope}[scale=1]
        \begin{axis}[
          width=\figurewidth, height=\figureheight,
          enlargelimits=false,
          axis on top,
          xtick={-1,0,1,2}, ytick={-1,0,1,2}, 
          xmin=-1.2, xmax=2.2, ymin=-1.2, ymax=2.2,          
          x axis line style = {-}, y axis line style = {-},
          major tick length = \ticklength
          ]
          \addplot graphics [xmin=-1.2,xmax=2.2,ymin=-1.2,ymax=2.2]
                   {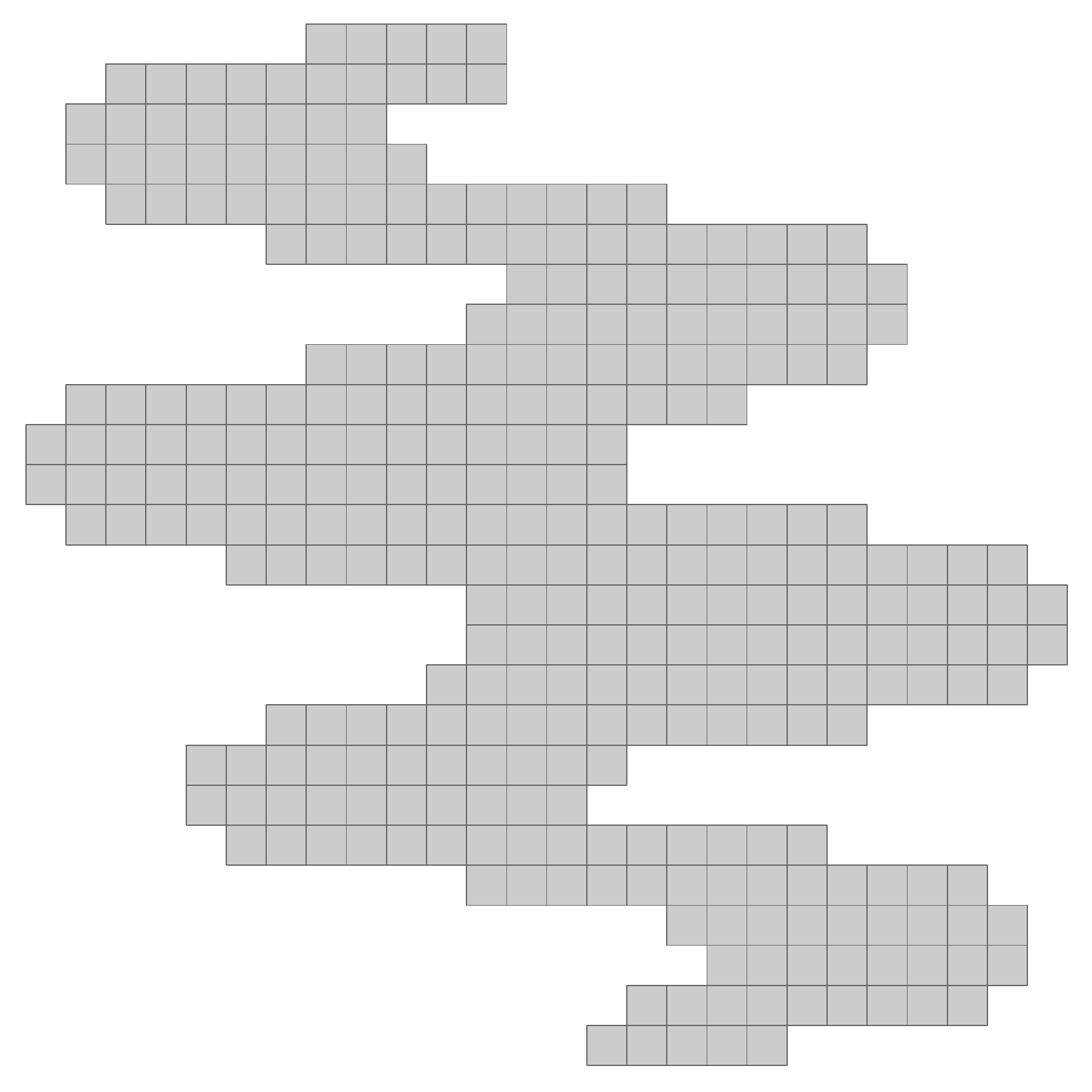};
        \end{axis}
      \end{scope}    
\end{tikzpicture}
  \end{minipage}
  \begin{minipage}[t]{0.45\textwidth}
    \vspace{0pt} % causes minipages being top-aligned
    \centering
\begin{tikzpicture} % n=2
      \begin{scope}[scale=1]
        \begin{axis}[
          width=\figurewidth, height=\figureheight,
          enlargelimits=false,
          axis on top,
          xtick={-1,0,1,2}, ytick={-1,0,1,2}, 
          xmin=-1.2, xmax=2.2, ymin=-1.2, ymax=2.2,          
          x axis line style = {-}, y axis line style = {-},
          major tick length = \ticklength
          ]
          \addplot graphics [xmin=-1.2,xmax=1.7,ymin=-1.2,ymax=1.7]
                   {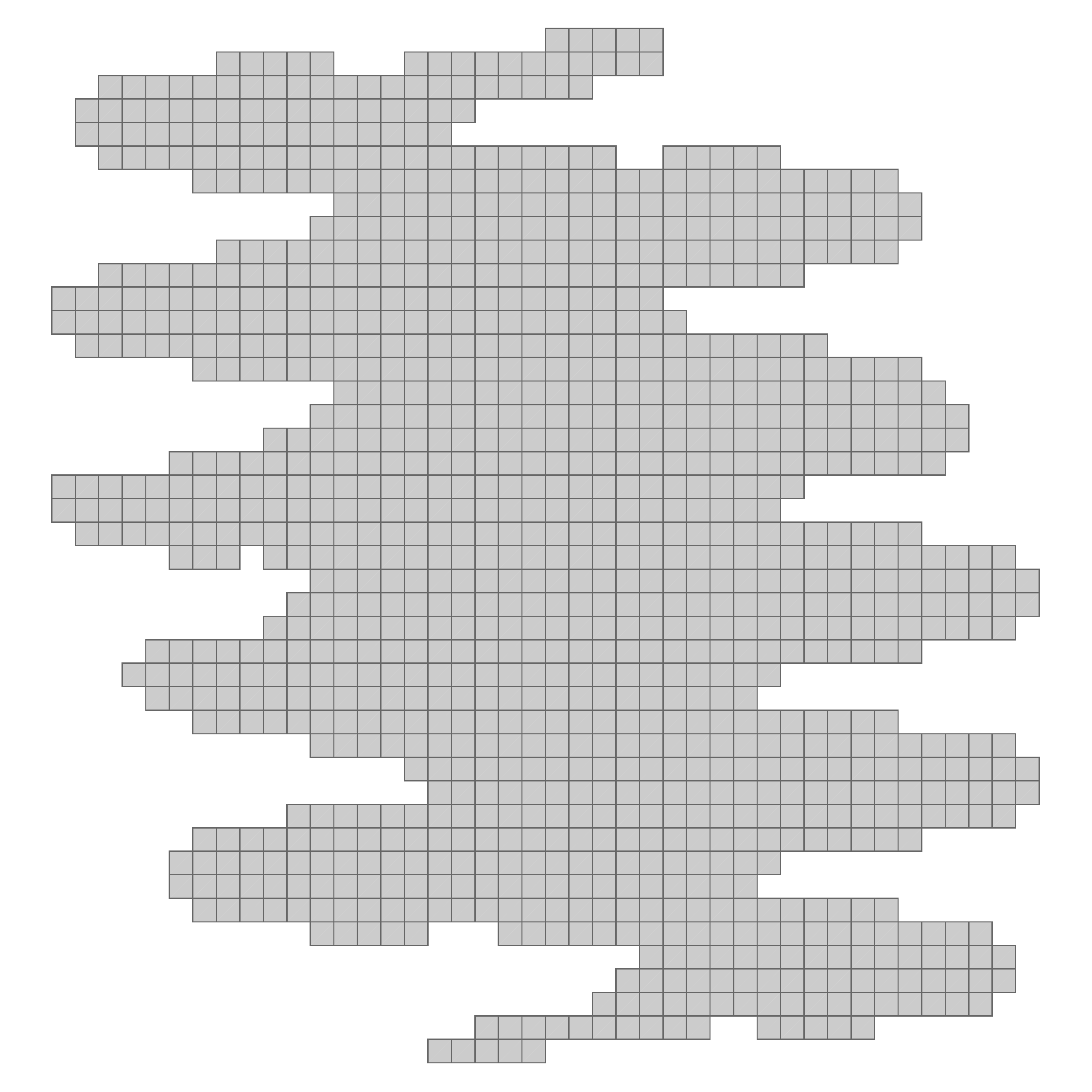};
        \end{axis}
      \end{scope}    
\end{tikzpicture}
  \end{minipage}\vspace{2mm}

 \begin{minipage}[t]{0.45\textwidth}
    \vspace{0pt} % causes minipages being top-aligned
    \centering
\begin{tikzpicture} % n=5
      \begin{scope}[scale=1]
        \begin{axis}[
          width=\figurewidth, height=\figureheight,
          enlargelimits=false,
          axis on top,
          xtick={-1,0,1}, ytick={-1,0,1}, 
          x axis line style = {-}, y axis line style = {-},
          major tick length = \ticklength
          ]
          \addplot graphics  [xmin=-1.2,xmax=1.7,ymin=-1.2,ymax=1.7]
                   {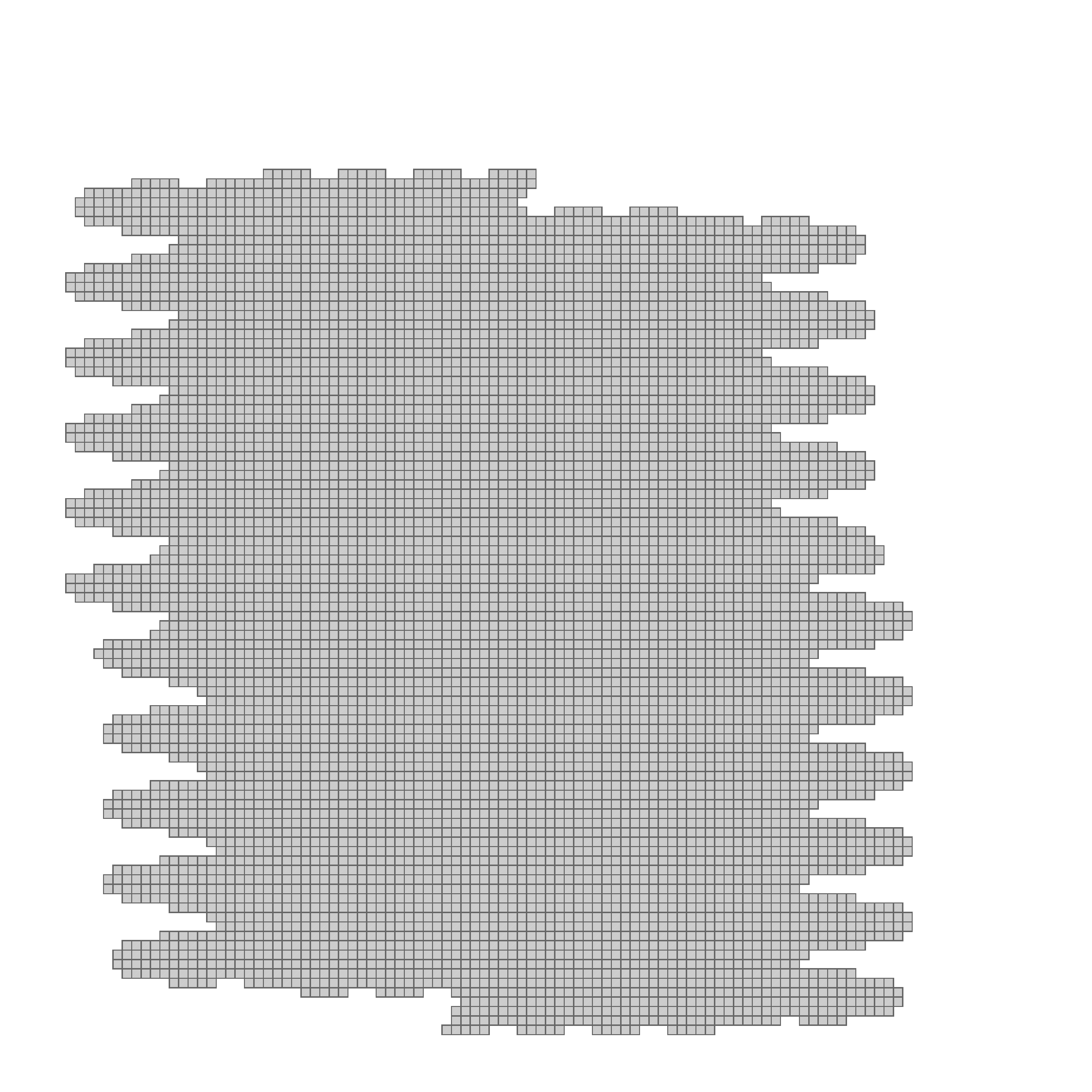};
        \end{axis}
      \end{scope}    
\end{tikzpicture}
  \end{minipage}
  \begin{minipage}[t]{0.45\textwidth}
    \vspace{0pt} % causes minipages being top-aligned
    \centering
\begin{tikzpicture} % n=10
      \begin{scope}[scale=1]
        \begin{axis}[
          width=\figurewidth, height=\figureheight,
          enlargelimits=false,
          axis on top,
          xtick={-1,0,1,2}, ytick={-1,0,1,2}, 
          xmin=-1.2, xmax=2.2, ymin=-1.2, ymax=2.2,          
          x axis line style = {-}, y axis line style = {-},
          major tick length = \ticklength
          ]
          \addplot graphics[xmin=-1.2,xmax=1.7,ymin=-1.2,ymax=1.7]
                   {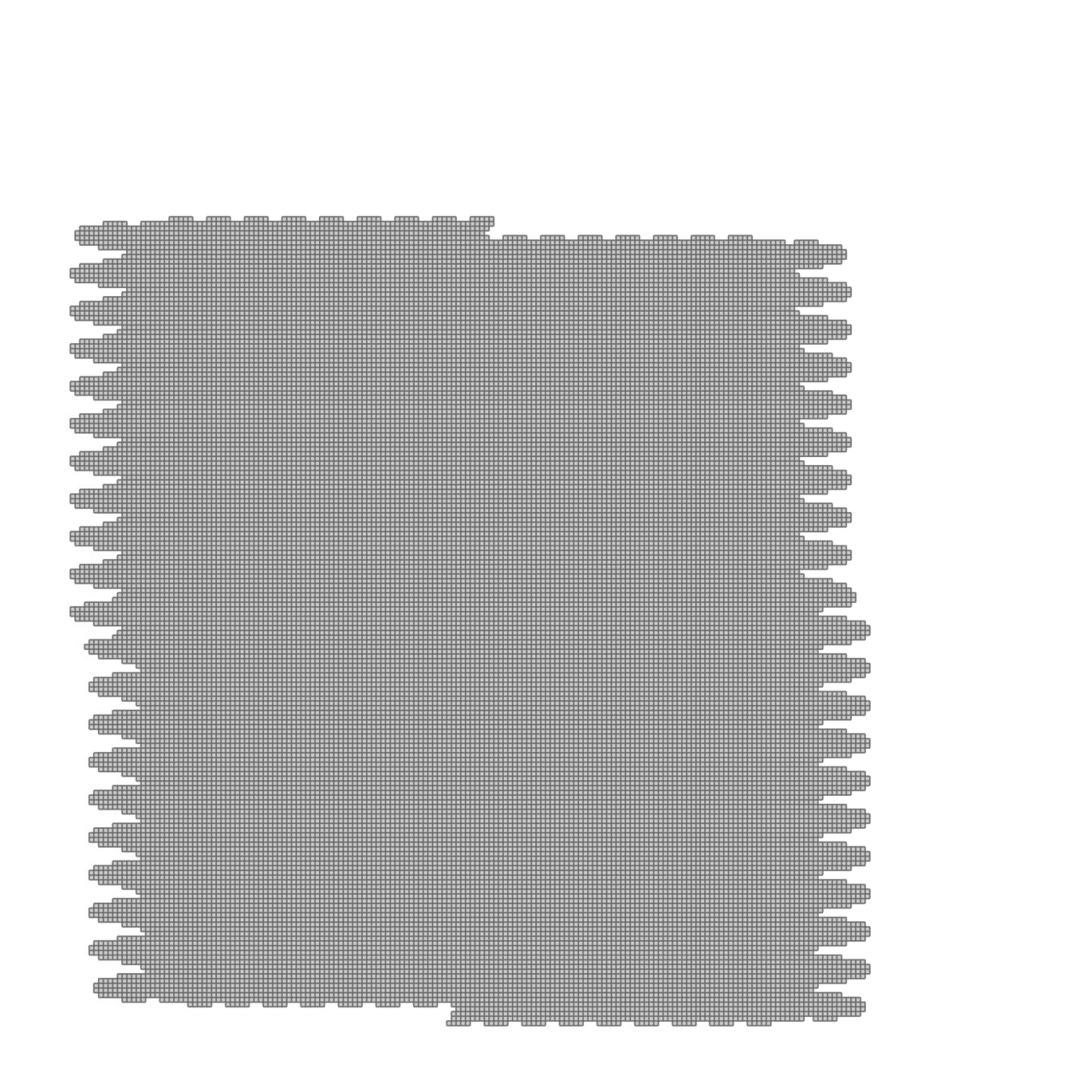};
        \end{axis}
      \end{scope}    
\end{tikzpicture}
  \end{minipage}
   \vspace{2mm}

  \begin{minipage}[t]{0.45\textwidth}
    \vspace{0pt} % causes minipages being top-aligned
    \centering
\begin{tikzpicture} % n=25
      \begin{scope}[scale=1]
        \begin{axis}[
          width=\figurewidth, height=\figureheight,
          enlargelimits=false,
          axis on top,
          xtick={-1,0,1,2}, ytick={-1,0,1,2}, 
          xmin=-1.2, xmax=2.2, ymin=-1.2, ymax=2.2,          
          x axis line style = {-}, y axis line style = {-},
          major tick length = \ticklength
          ]
          \addplot graphics [xmin=-1.2,xmax=1.7,ymin=-1.2,ymax=1.7]
                   {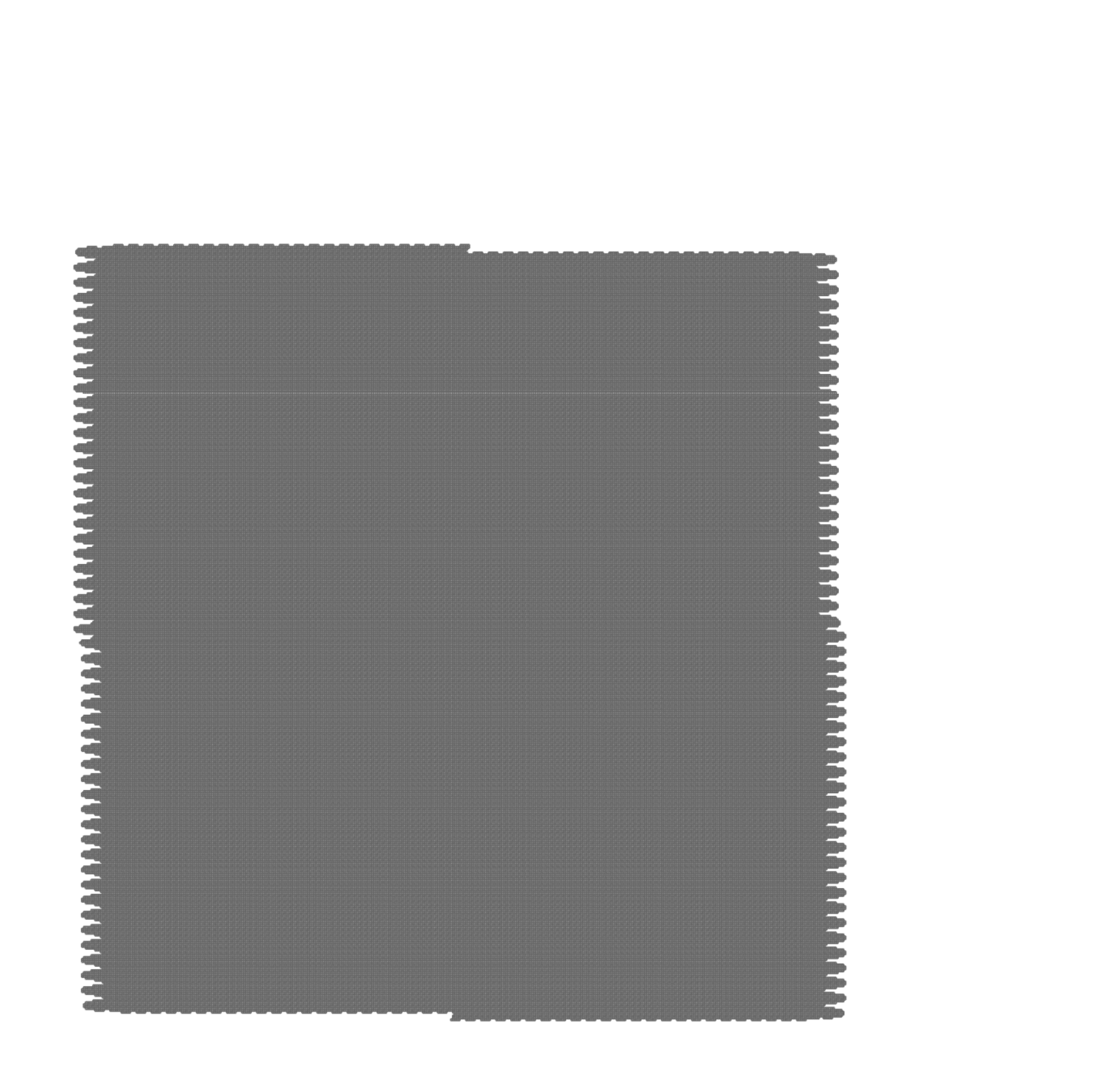};
        \end{axis}
      \end{scope}    
\end{tikzpicture}
  \end{minipage}
  \begin{minipage}[t]{0.45\textwidth}
    \vspace{0pt} % causes minipages being top-aligned
    \centering
\begin{tikzpicture} % n=50
      \begin{scope}[scale=1]
        \begin{axis}[
          width=\figurewidth, height=\figureheight,
          enlargelimits=false,
          axis on top,
          xtick={-1,0,1,2}, ytick={-1,0,1,2}, 
          xmin=-1.2, xmax=2.2, ymin=-1.2, ymax=2.2,          
          x axis line style = {-}, y axis line style = {-},
          major tick length = \ticklength
          ]
          \addplot graphics [xmin=-1.2,xmax=1.7,ymin=-1.2,ymax=1.7]
                   {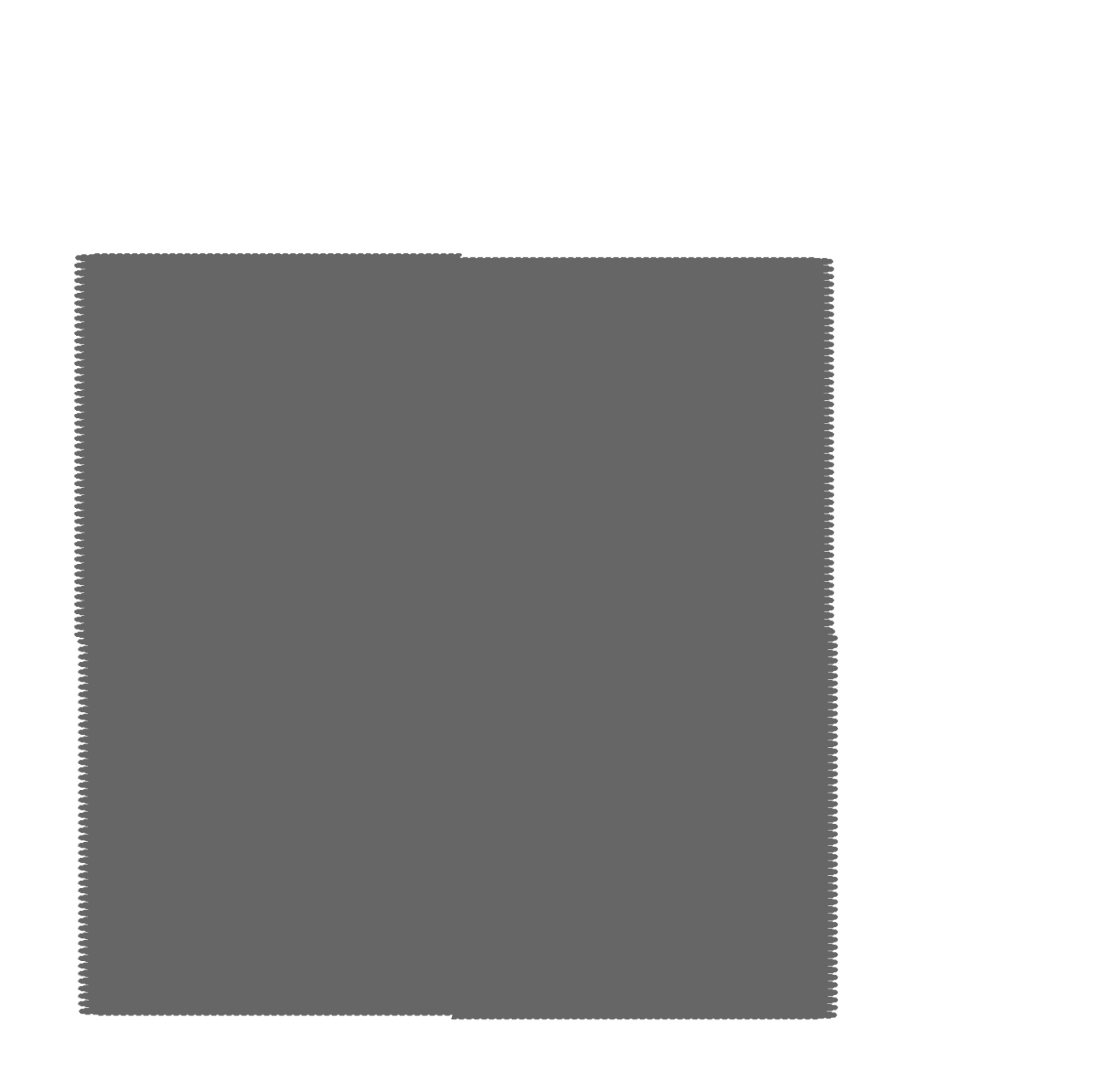};
        \end{axis}
      \end{scope}    
\end{tikzpicture}
  \end{minipage}\vspace{2mm}

  \begin{minipage}[t]{0.45\textwidth}
    \vspace{0pt} % causes minipages being top-aligned
    \centering
\begin{tikzpicture} % n=100
      \begin{scope}[scale=1]
        \begin{axis}[
            width=\figurewidth, height=\figureheight,
          enlargelimits=false,
          axis on top,
          xtick={-1,0,1,2}, ytick={-1,0,1,2}, 
          xmin=-1.2, xmax=2.2, ymin=-1.2, ymax=2.2,          
          x axis line style = {-}, y axis line style = {-},
          major tick length = \ticklength
          ]
          \addplot graphics [xmin=-1.2,xmax=1.7,ymin=-1.2,ymax=1.7]
                   {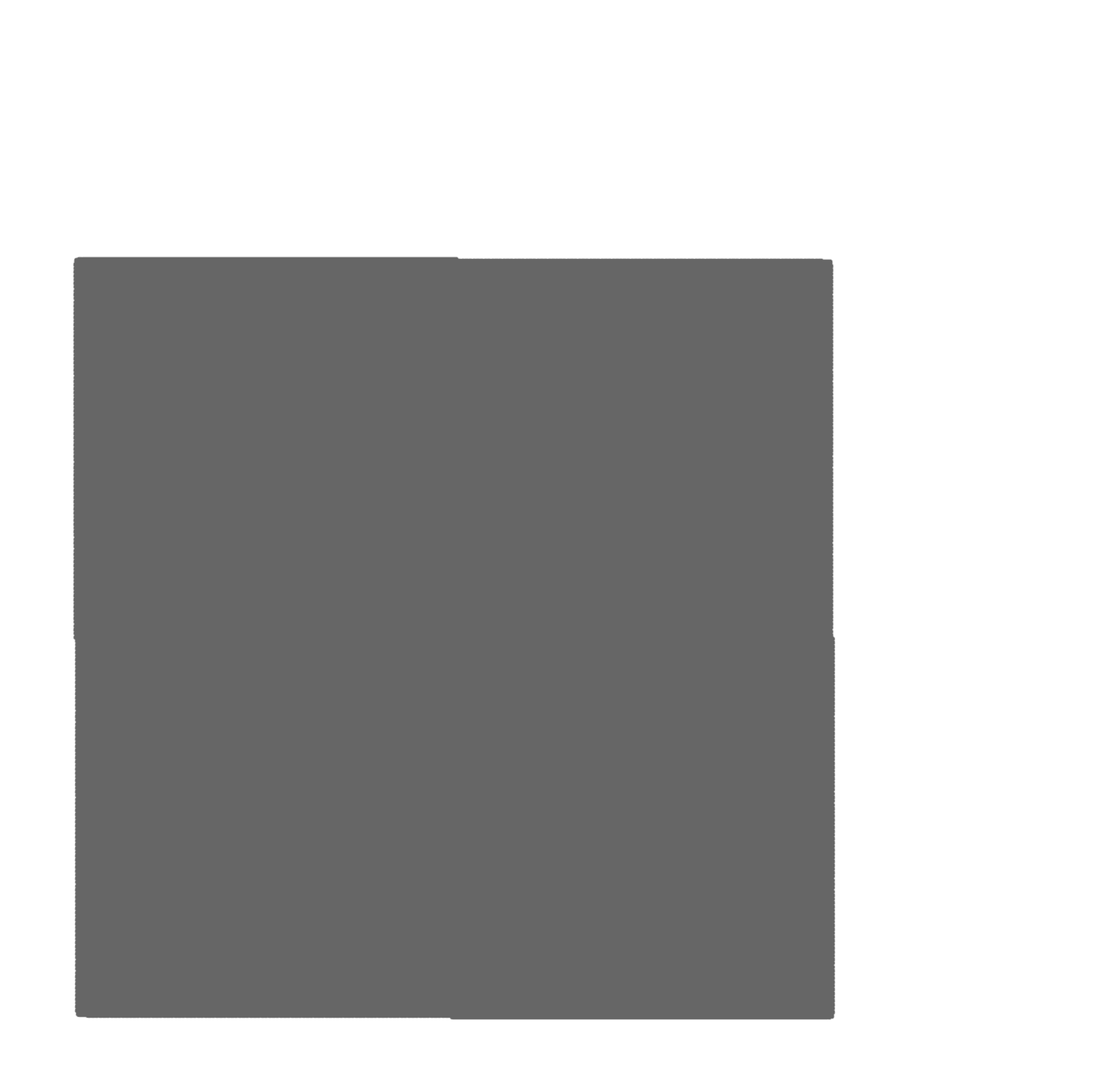};
        \end{axis}
      \end{scope}    
\end{tikzpicture}
  \end{minipage}
  \begin{minipage}[t]{0.45\textwidth}
    \vspace{0pt} % causes minipages being top-aligned
    \centering
\begin{tikzpicture} % n=200
      \begin{scope}[scale=1]
        \begin{axis}[
          width=\figurewidth, height=\figureheight,
          enlargelimits=false,
          axis on top,
          xtick={-1,0,1,2}, ytick={-1,0,1,2}, 
          xmin=-1.2, xmax=2.2, ymin=-1.2, ymax=2.2,          
          x axis line style = {-}, y axis line style = {-},
          major tick length = \ticklength
          ]
          \addplot graphics [xmin=-1.2,xmax=1.7,ymin=-1.2,ymax=1.7]
                   {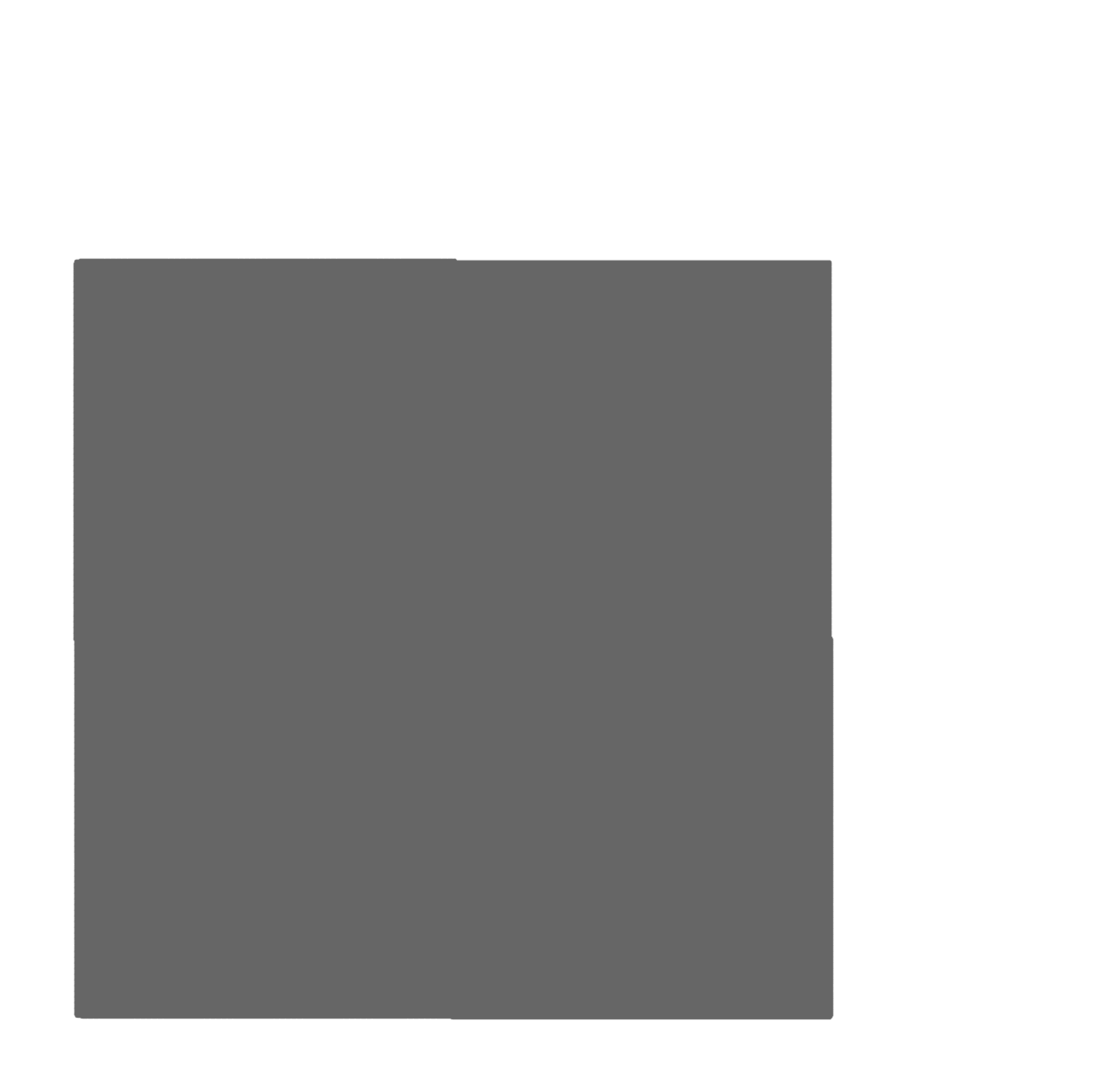};
        \end{axis}
      \end{scope}    
\end{tikzpicture}
  \end{minipage}
  \caption{Approximations $Q^*_{k,n}$ for the rotation set of the map $F_{1,1}$ with $k=8$ and $n = 1,2,5,10,25,50,100,200$ (from top left to bottom right).}
    \label{f.iterates_for_F11_1}
\end{center}
 \end{figure}

\bigskip
\noindent 
Zooming in (Fig. \ref{f.F11_zoom_ul}) on the boundary of $Q_{8,100}^*$ and
$Q_{8,200}^*$ of Figure~\ref{f.iterates_for_F11_1} reveals the difference that
exists between these approximations and $\rs{F_{1,1}}=[0,1]^2$, which is of a
magnitude of $10^{-2}$ (and hence significantly smaller than the theoretical
error bound in (\ref{E_main_err})).  By Lemma \ref{Prop_mainalg_1}, the
$\nicefrac{2\sqrt{2}}{n}$-neighbourhood of $Q_{k,n}^\ast$ covers the rotation~set.

 %%%%%%%
 %%% figure of zoom
 %%%%%%%
 
\begin{figure}[!ht]
\rule{0pt}{10pt}
   \labelsize
   \setlength\figureheight{6.5cm} 
   \setlength{\ticklength}{1.5mm}  
   \begin{center}
   \begin{minipage}[t]{0.49\textwidth}
      \centering \vspace{0pt} % causes minipages being top-aligned
      \begin{tikzpicture} % zoom in for n=100
         \begin{axis}[
            axis equal image,        
            height=\figureheight,
            enlargelimits=false,
            axis on top,
            major tick length = \ticklength,
            %every axis/.append style={black},
            every x tick/.style={black},
            every y tick/.style={black},
            ytick={1.04,1,0.96,0.92,0.88}, xtick={-1.04,-1,-0.96,-0.92,-0.88}, 
            x axis line style = {-},  y axis line style = {-}
            ]
            \addplot graphics [xmin=-1.06,xmax=-0.85,ymin=0.85,ymax=1.06]
                         {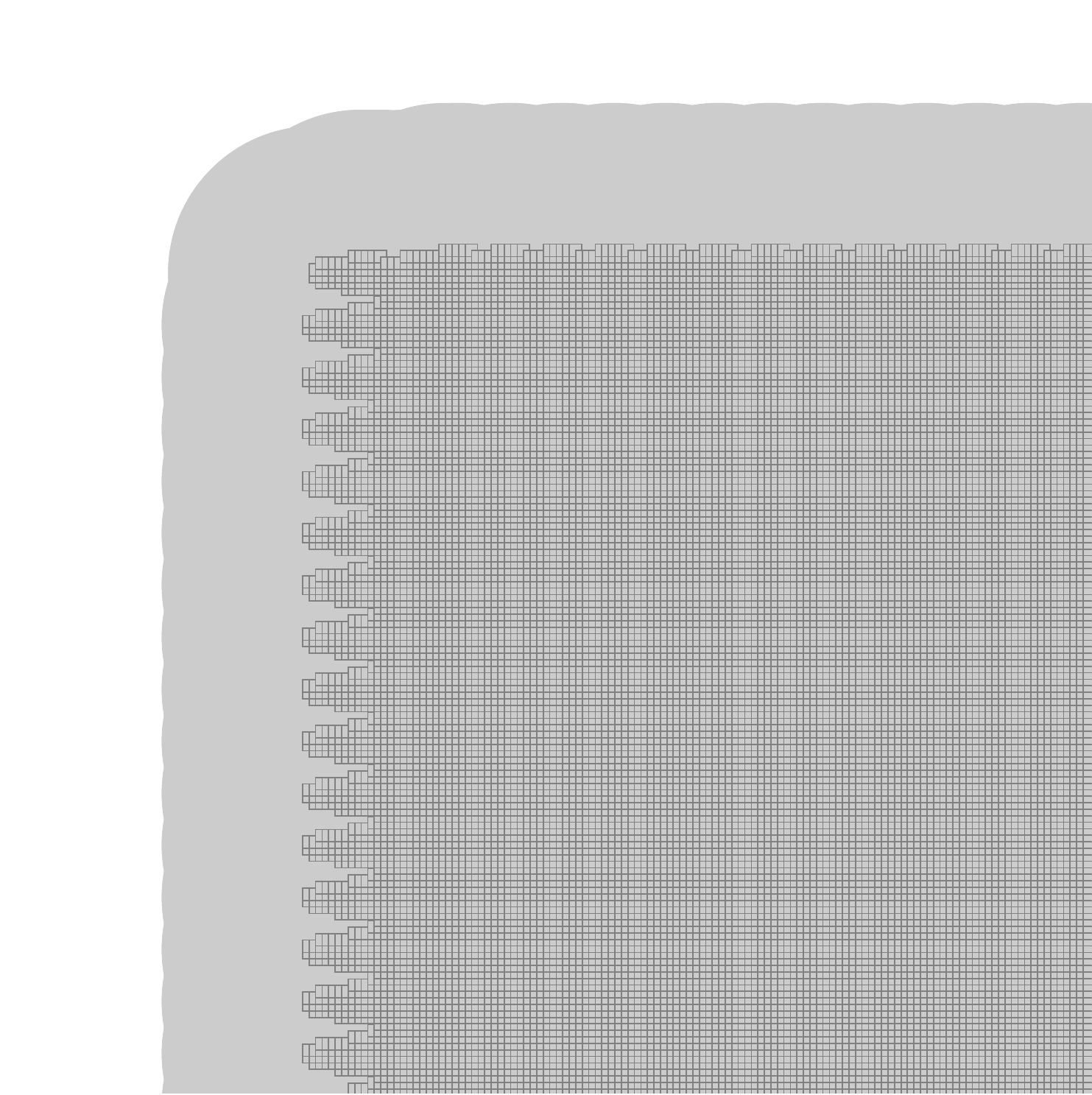};
            \draw[densely dashed, color=black] 
                     (axis cs:-1,0.85) -- (axis cs:-1,1) -- (axis cs: -0.85,1); 
         \end{axis}   
      \end{tikzpicture}
   \end{minipage}
   %\hspace{1.5cm}
   \begin{minipage}[t]{0.49\textwidth}
      \centering \vspace{0pt} % causes minipages being top-aligned
      \begin{tikzpicture} % zoom in for n=200
        \begin{axis}[ axis equal image, height=\figureheight,
          enlargelimits=false, axis on top, major tick length = \ticklength,
          every x tick/.style={black}, every y tick/.style={black}, xticklabel
          style={/pgf/number format/precision=4}, yticklabel style={/pgf/number
            format/precision=4}, ytick={1.04,1,0.96,0.92,0.88},
          xtick={-1.04,-1,-0.96,-0.92,-0.88}, x axis line style = {-}, y axis
          line style = {-} ]
            \addplot graphics [xmin=-1.06,xmax=-0.85,ymin=0.85,ymax=1.06]
                         {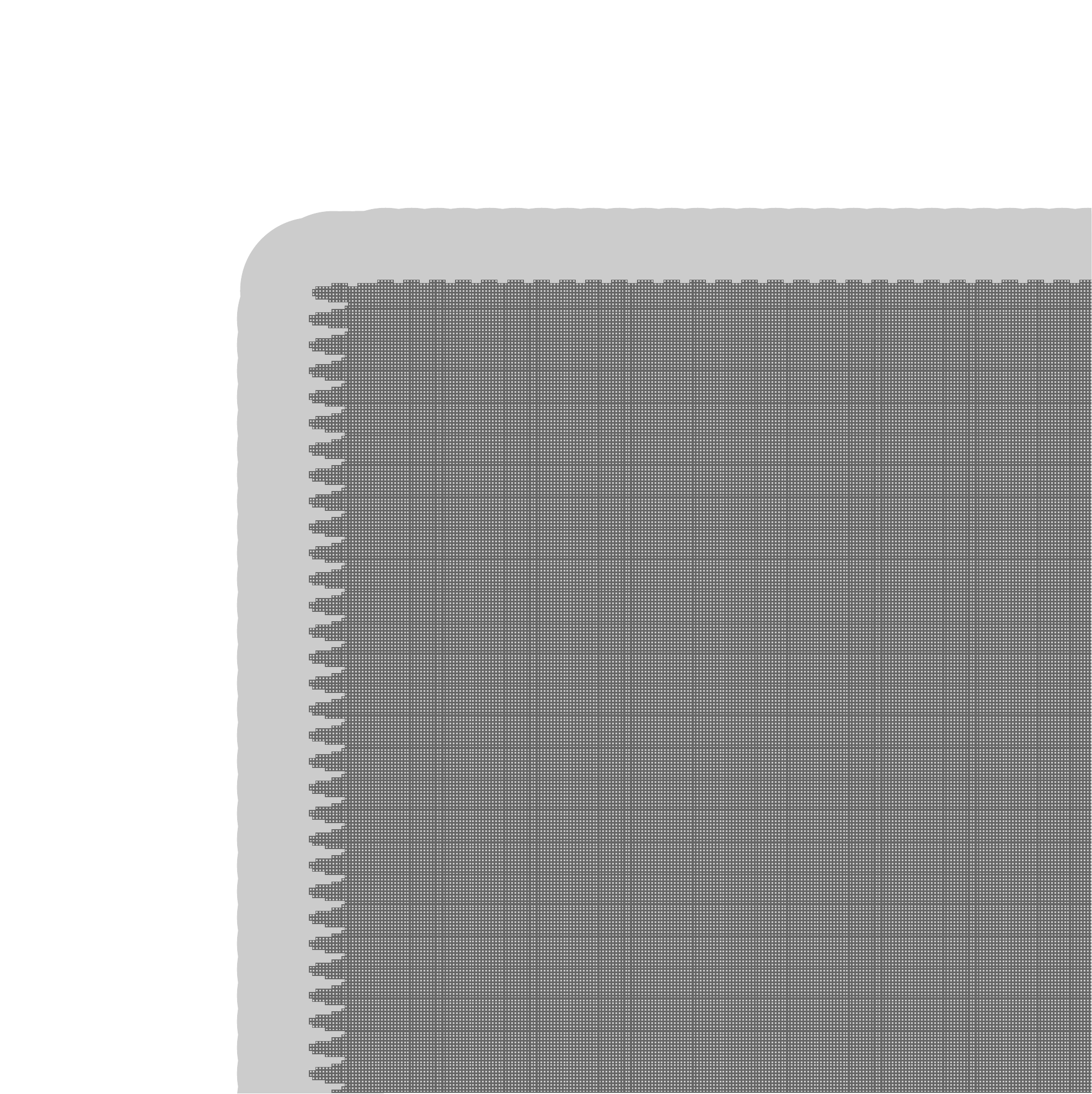};
            \draw[densely dashed, color=black] 
                     (axis cs:-1,0.85) -- (axis cs:-1,1) -- (axis cs: -0.85,1);
         \end{axis}   
      \end{tikzpicture} 
   \end{minipage}
   \end{center}
   %\captionsetup{width=0.9\textwidth}
   \caption{Zoom on top left area of the approximations $Q_{8,100}^\ast$ (left) and
     $Q_{8,200}^\ast$ (right) for the rotation set of the map $F_{1,1}$.}
  \label{f.F11_zoom_ul}
\end{figure}

\bigskip
\noindent 
As further examples, we consider the maps $F_{\nicefrac{1}{2},\nicefrac{1}{2}},F_{\nicefrac{3}{5},\nicefrac{3}{5}}$ and
$F_{\nicefrac{3}{4},1}$. In the first case, we still have an a priori lower bound for the
rotation set: $\rs{F_{\nicefrac{1}{2},\nicefrac{1}{2}}}$ contains the square spanned by the points $(\pm \nicefrac{1}{2},0)$ and $(0,\pm \nicefrac{1}{2})$, since these rotation vectors are realised by the
two-periodic points $(0,\nicefrac{1}{4}),(\nicefrac{1}{4},0)$,$(0,\nicefrac{3}{4})$ and $(\nicefrac{3}{4},0)$. The numerical
approximation in Figure \ref{f.further_RS} indicates that this square indeed is
the rotation set of $F_{\nicefrac{1}{2},\nicefrac{1}{2}}$ (recall here that our algorithm never
underestimates).

\begin{figure}[!ht]
   \labelsize
   \setlength\figureheight{4cm} 
   \setlength{\ticklength}{1mm} 
   \begin{center}
     \begin{tikzpicture} % alpha = beta = 0.5
        \begin{axis}[
          axis equal image,
          height = \figureheight,
          xmin = -0.6, xmax = 0.6, ymin = -0.6, ymax = 0.6,
          axis on top,
          major tick length = \ticklength,
          xtick={-0.5,0,0.5}, ytick={-0.5,0,0.5},    
          x axis line style = {-}, y axis line style = {-}
          ]
          \addplot graphics
                    [xmin=-0.55,xmax=0.75,ymin=-0.55,ymax=0.75]
                   {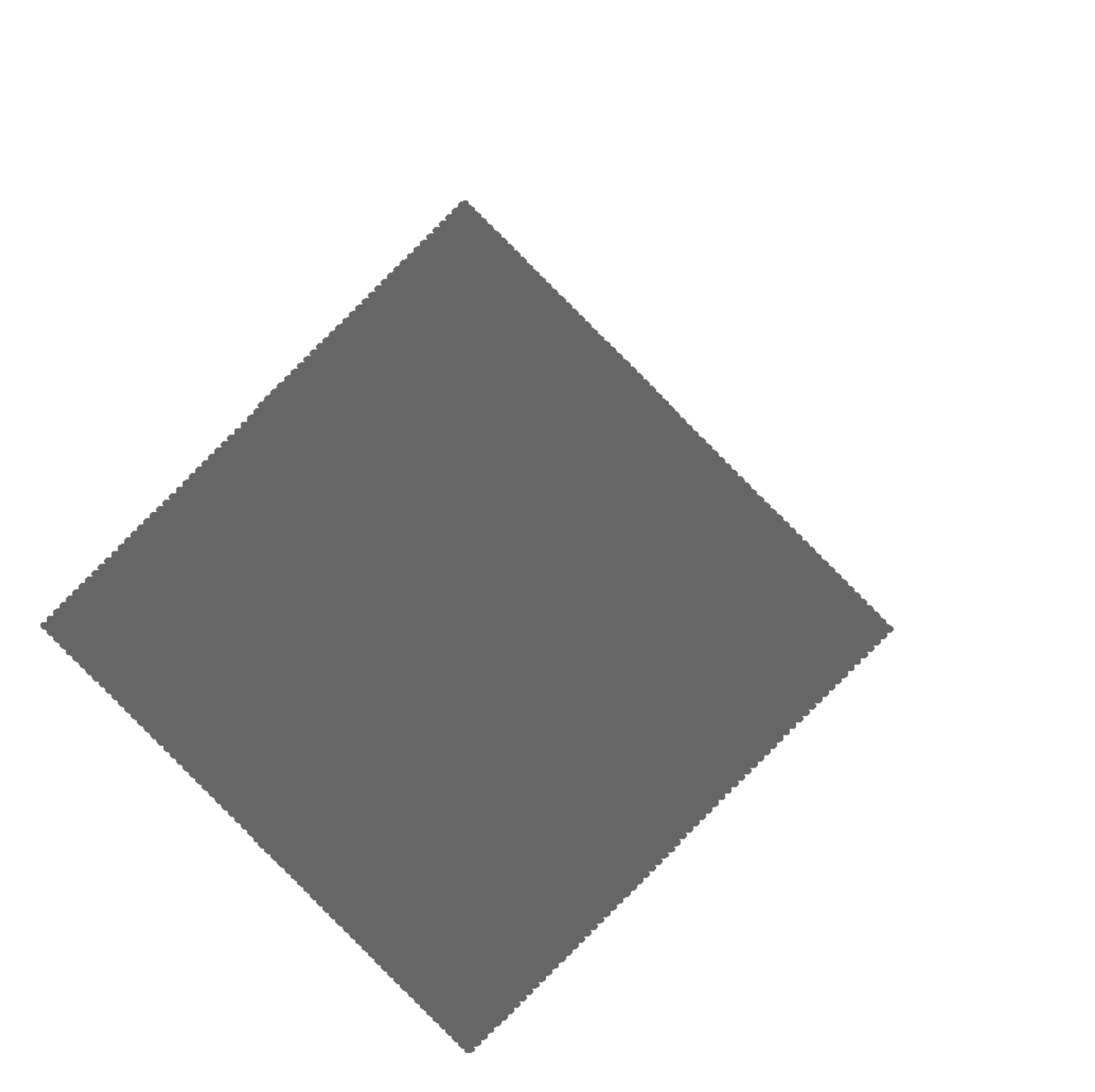};
        \end{axis}   
     \end{tikzpicture} 
     \hspace{2mm}
      \begin{tikzpicture} % alpha = 1, beta = 1/4 
        \begin{axis}[
          axis equal image,
          height = 1.298\figureheight,
          major tick length = \ticklength,
          axis on top,
          xtick={-1,-0.5,0,0.5,1}, ytick={-0.5,0,0.5}, 
          xmin = -1.1, xmax = 1.1, ymin = -0.6, ymax = 0.6,        
          x axis line style = {-}, y axis line style = {-},
          yticklabel pos=right
          ]
          \addplot graphics
                   [xmin=-1.05,xmax=1.05,ymin=-0.3,ymax=0.3]
                   {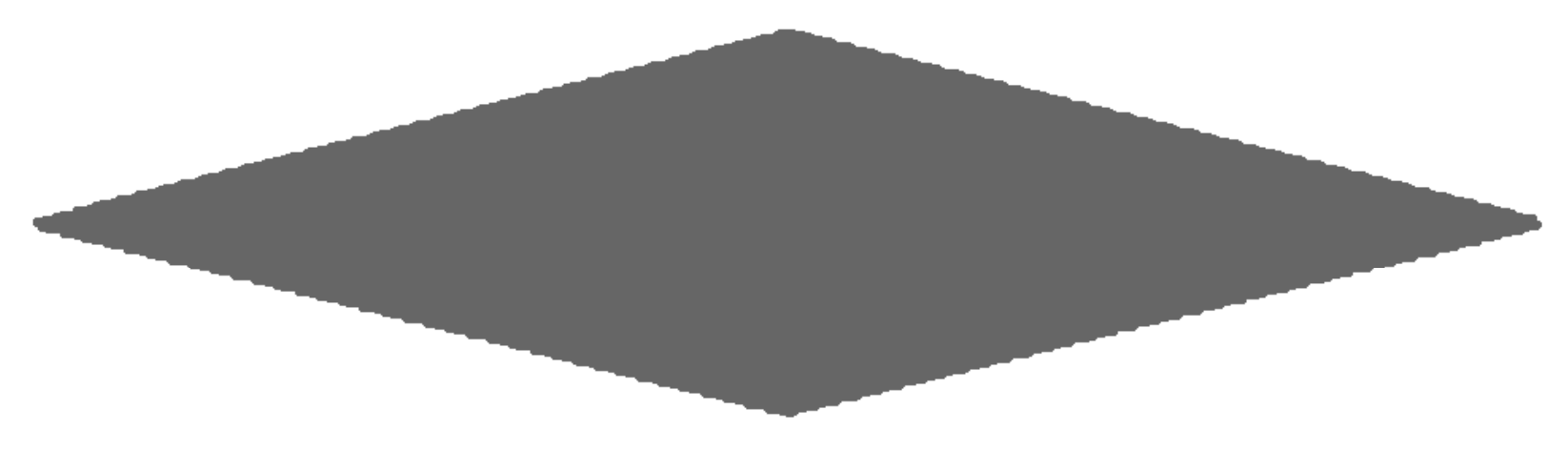};
        \end{axis}   
      \end{tikzpicture}  
      
      \smallskip
     \begin{tikzpicture} % alpha = beta = 0.6
        \begin{axis}[
          height = 1.6\figureheight,
          axis equal image,
          xmin=-0.6,xmax=0.6,ymin=-1.2,ymax=1.2,
          axis on top,
          major tick length = \ticklength,
          xtick={-0.5,0,0.5}, ytick={-1,-0.5,0,0.5,1},          
          x axis line style = {-}, y axis line style = {-}
          ]
          \addplot graphics
                    [xmin=-0.55,xmax=.55,ymin=-0.55,ymax=0.55]
                   {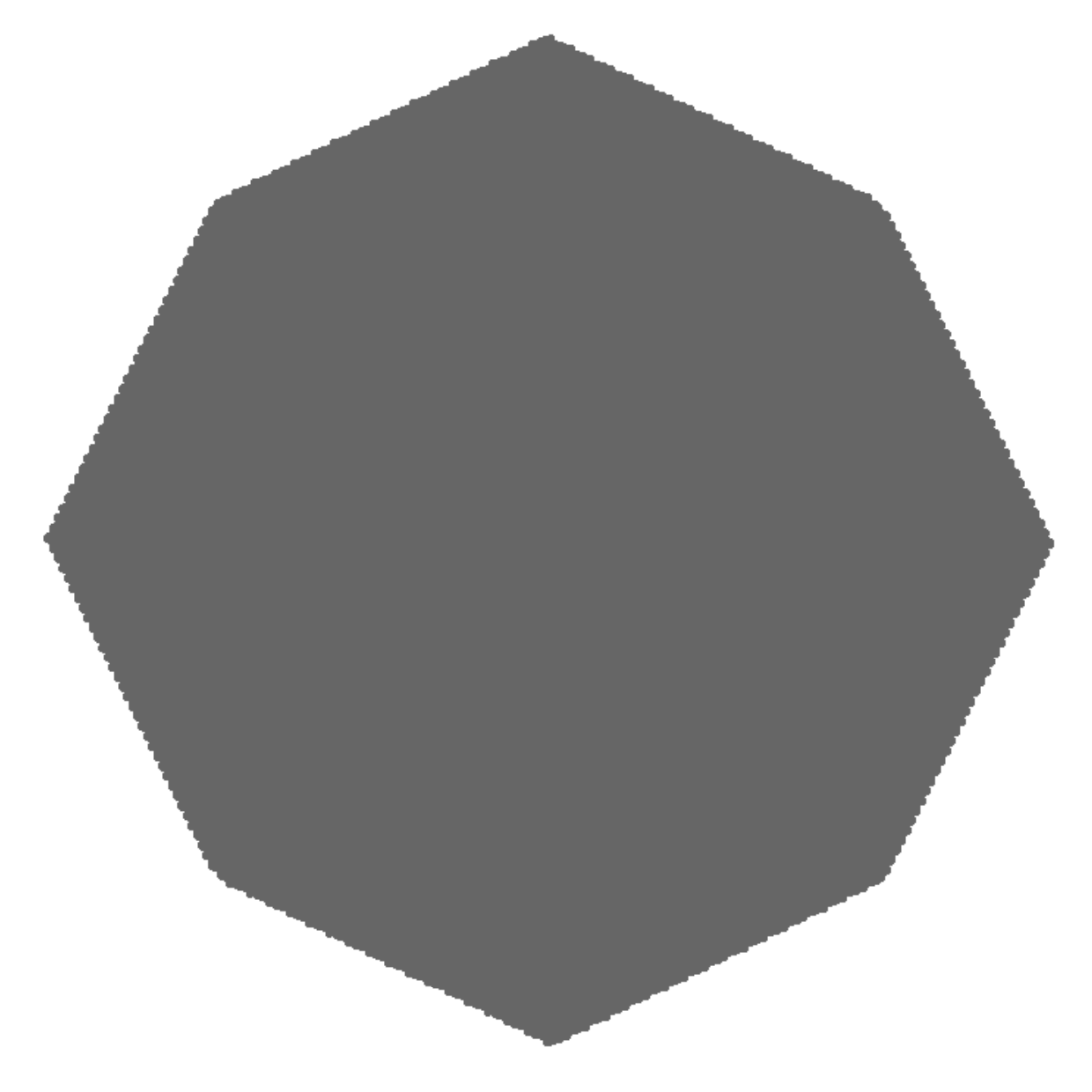};
        \end{axis}   
     \end{tikzpicture}    
     \hspace{2mm}   
     \begin{tikzpicture} % alpha = 3/4, beta = 1
         \begin{axis}[
          height = 1.607\figureheight,
          axis equal image,
          xmin=-1.1,xmax=1.1,ymin=-1.2,ymax=1.2,
          axis on top,
          major tick length = \ticklength,
          xtick={-1,-0.5,0,0.5,1}, ytick={-1,-0.5,0,0.5,1},      
          x axis line style = {-}, y axis line style = {-},
          yticklabel pos=right
          ]
          \addplot graphics
                   [xmin=-0.85,xmax=0.85,ymin=-1.1,ymax=1.1]
                   {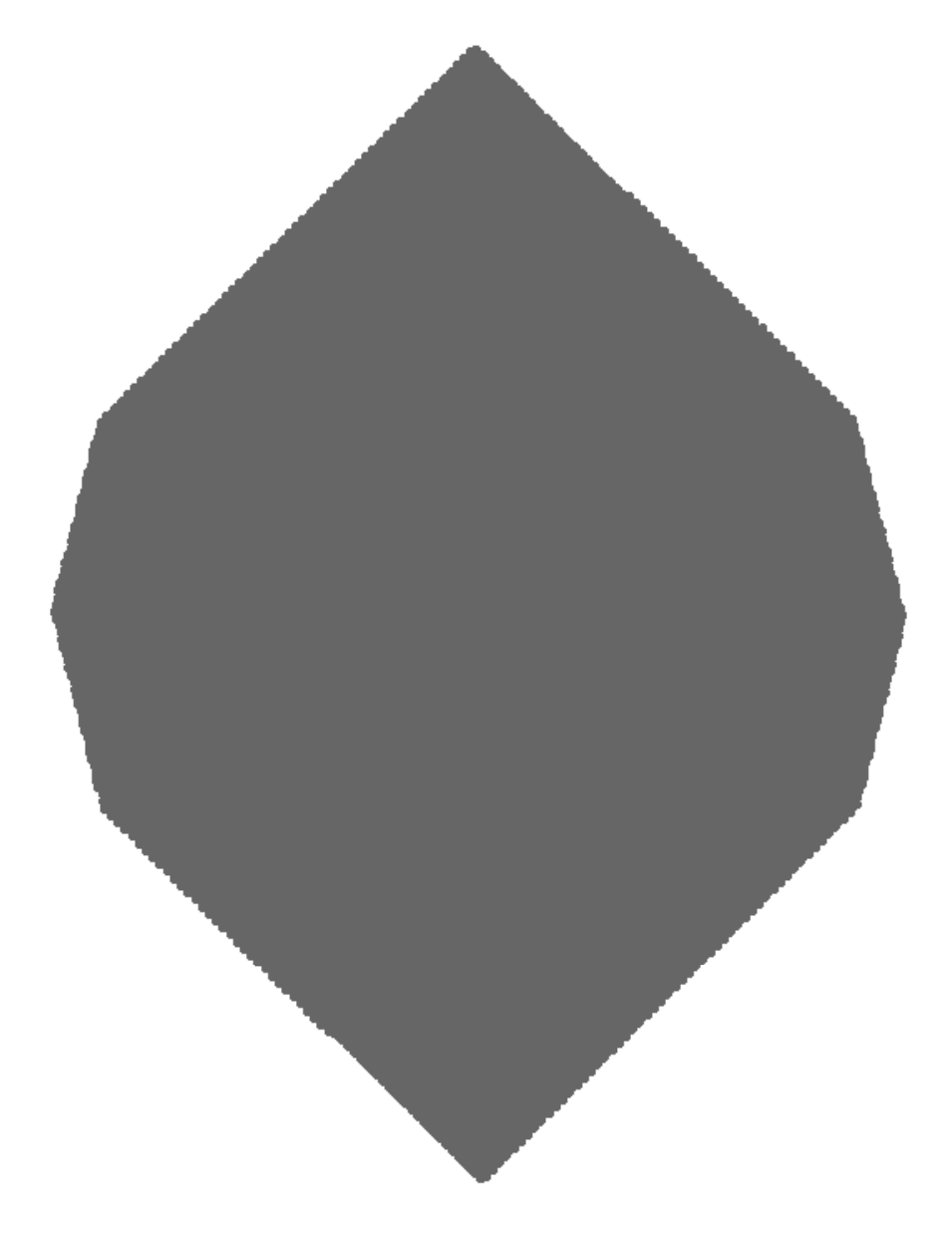};
        \end{axis}   
  \end{tikzpicture}          
   \end{center}   
   \caption{Approximations $Q_{k,n}^*$ for the rotation sets of the maps  $F_{\nicefrac{1}{2},\nicefrac{1}{2}}, F_{1,\nicefrac{1}{4}},
   F_{\nicefrac{3}{5},\nicefrac{3}{5}}$ and $F_{\nicefrac{3}{4},1}$ (with $k=50, 16, 50,45$ and $n=130, 140, 100,80$ from top left to bottom right)
  \label{f.further_RS} }
\end{figure}

\bigskip
\noindent 
In the examples $F_{1,1}$ and $F_{\nicefrac{1}{2},\nicefrac{1}{2}}$ above, the vertices of the polygonal
rotation set are realised by periodic orbits of very low period (1, respectively
2). As discussed in \cite{Guiheneuf2015RotationSetComputation}, such rotation
sets can still be accurately predicted by conventional direct approaches, but
these tend to fail if the vertices correspond to periodic points of higher
periods. For this reason, we next consider the map $G = g_3 \circ g_2 \circ g_1:
\mathbb{R}^2 \longrightarrow \mathbb{R}^2$ with
\begin{align*}
  g_1(x,y) &= \left( x,\ y + \textstyle{\frac{1}{8}} \sin (5 \cdot 2 \pi x) \right), \\
  g_2(x,y) &=\left( x + \textstyle{\frac{2}{5}} \sin (8 \cdot 2 \pi y), \ y \right)
   \text{ and } \\
  g_3(x,y) &=\left( x -  \textstyle{\frac{1}{5}},\ y + \textstyle{\frac{2}{8}} \right) \ . 
\end{align*}
%\begin{align*} % alternatively
%  g_1(x,y) &= \left( x,\ y + \frac{1}{8} \sin (5 \cdot 2 \pi x) \right), 
%  g_2(x,y) =\left( x + \frac{2}{5} \sin (8 \cdot 2 \pi y), \ y \right) 
%   \text{ and } \\
%  g_3(x,y) &=\left( x - \frac{1}{5},\ y + \frac{2}{8}\right). 
%\end{align*}
%\begin{align*} %alternatively
%  g_1(x,y) \coloneqq 
%    \begin{pmatrix}
%      x \\
%      y + \frac{1}{8} \sin (5 \cdot 2 \pi x)
%    \end{pmatrix}, 
%  g_2(x,y) \coloneqq
%    \begin{pmatrix}
%      x + \frac{2}{5} \sin (8 \cdot 2 \pi y) \\
%      y
%    \end{pmatrix}, % \text{ and } 
%  g_3(x,y) \coloneqq 
%    \begin{pmatrix}
%      x - \frac{1}{5} \\
%      y + \frac{2}{8}
%    \end{pmatrix}. 
%\end{align*}
The related rotation set can be determined analytically as the rectangle $\rs{G}
= [-\nicefrac{3}{5},\nicefrac{1}{5}] \times [\nicefrac{1}{8},\nicefrac{3}{8}]$ since its vertices correspond to elliptic
periodic orbits of period $40$ (those of the points
$(\nicefrac{3}{20},\nicefrac{3}{32}),(\nicefrac{1}{20},\nicefrac{3}{32}),
(\nicefrac{1}{20},\nicefrac{1}{32})$ and $(\nicefrac{3}{20},\nicefrac{1}{32})$). Figure
\ref{f.high_period} shows the approximate rotation set $Q_{60,130}^*$ for $G$.

\begin{figure}[!ht]
   \labelsize
   \setlength\figureheight{8cm} 
   \setlength{\ticklength}{1mm} 
   \begin{center}
     \begin{tikzpicture} 
        \begin{axis}[
          height = \figureheight,
          axis equal image,
          axis on top,
          xtick={-0.6,-0.4,-0.2,0,0.2}, ytick={0.125,0.250,0.375}, 
          xmin=-0.65,xmax=0.27,ymin=0.07,ymax=0.43,
          major tick length = \ticklength,
          xticklabel style={/pgf/number format/precision=4},
          yticklabel style={/pgf/number format/precision=4},             
          x axis line style = {-}, y axis line style = {-}
          ]
          \addplot graphics
                   [xmin=-0.63,xmax=0.23,ymin=0.1,ymax=0.4]
                   {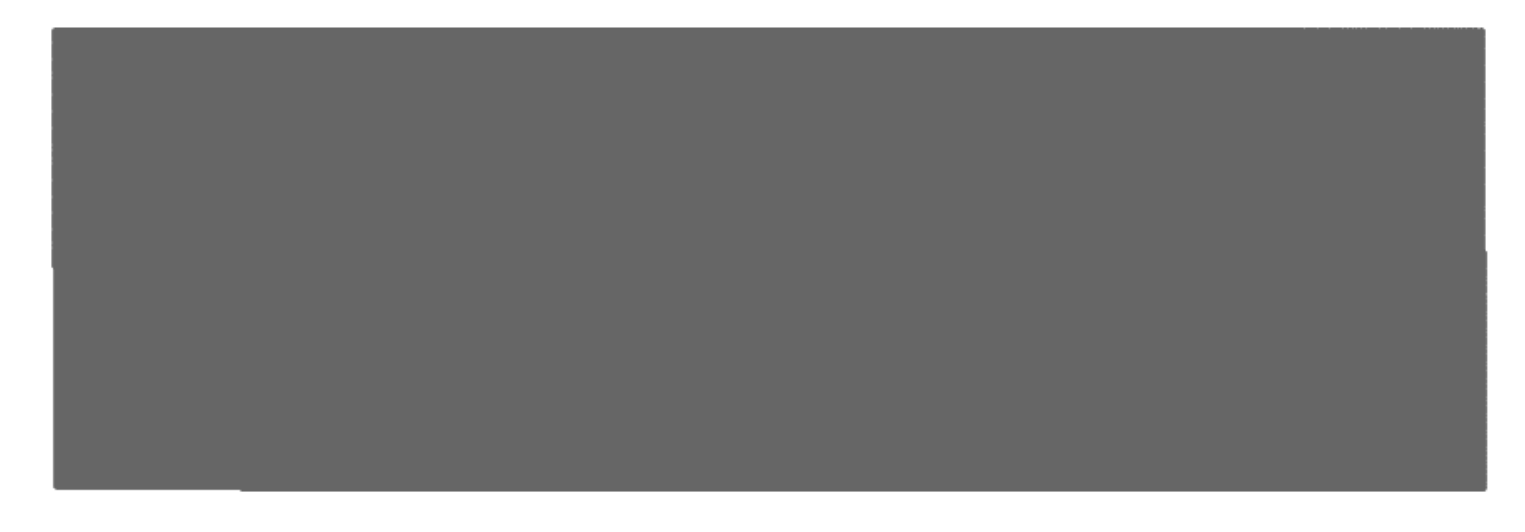};
        \end{axis}   
    \end{tikzpicture} 
  \end{center}
  \captionsetup{width=\textwidth, margin={0pt,2pt}}
  \caption{Approximation $Q_{60,130}^\ast$ for the rotation set of the map $G$.}
  \label{f.high_period}
\end{figure}

\clearpage
\noindent
Finally, we consider a slightly perturbed version $\bar{F} = R \circ F$ of the
above examples by introducing a slight additional rotation \[ R:\mathbb{R}^2 \to
\mathbb{R}^2, (x,y) \mapsto (x+r_1,y+r_2), r_1,r_2\in \mathbb{R} \ . \] In Table \ref{t.table} we collect the specific parameter values for both $k$ and $n$ and the perturbations $r_1$ and $r_2$, on which we base our approximations of the related rotation sets of the perturbed maps.

Although it is difficult to check rigorously, we expect that for small perturbations these
modifications should not alter the rotation sets of the above examples due to
the generic structural stability of the dynamics
\cite{Passeggi2013RationalRotationSets}. Moreover, we expect that the vertices
of the rotation sets are still realised by periodic orbits that lie close to the
original ones. This fact could in principle be checked by a qualitative index
argument. However, we refrain from going into detail and content ourselves with
the numerical confirmation of the stability provided by Figure
\ref{f.further_RS_pert}.

\bigskip
\begin{table}[!ht]
\centering
\begin{tabular}{|c|c|c|c|c|c|c|}
\hline
         & $\bar{F}_{\frac{1}{2},\frac{1}{2}}$ 
         &  $\bar{G}$ 
         &  $\bar{F}_{1,\frac{1}{4}}$ 
         &  $\bar{F}_{\frac{3}{5},\frac{3}{5}}$ 
         & $\bar{F}_{\frac{3}{4},1}$ 
         & $\bar{F}_{1,1}$ \rule{0pt}{13pt} \\  
          \hline
 $k$    &$50$      &$60$     &$16$       &$50$    &$45$      &$8$ \\
 $n$    & $130$   &$130$   &$140$    &$100$   &$80$      &$100$ \\
 $r_1$ & $0.012$&$0.008$& $0.012$&$0.01$  &$0.002$ &$0.022$ \\
 $r_2$ & $0.014$&$0.001$&$0.002$ &$0.011$&$0.013$ &$0.015$ \\
 \hline
\end{tabular}

\rule{0pt}{7pt}
\caption{Parameter values for the approximations shown in Figure \ref{f.further_RS_pert}.}
\label{t.table}
\end{table}

%%%%%%%%%%%%%%%%%%%%%%%%%%
%%%%%%% other RS pert %%%%%%%%%%%

\begin{figure}[!ht]
   \labelsize
   \setlength\figureheight{3.8cm} 
   \setlength{\ticklength}{1mm} 
   \begin{center}

 \setlength{\tabcolsep}{0.5mm}
  \begin{tabular}{ r r r }
     \begin{tikzpicture} % alpha = beta = 0.5
        \begin{axis}[
          axis equal image,
          height = \figureheight,
          xmin = -0.6, xmax = 0.6, ymin = -0.6, ymax = 0.6,
          axis on top,
          major tick length = \ticklength,
          xtick={-0.5,0,0.5}, ytick={-0.5,0,0.5},    
          x axis line style = {-}, y axis line style = {-}
          ]
          \addplot graphics
                    [xmin=-0.55,xmax=0.55,ymin=-0.55,ymax=0.55]
                   {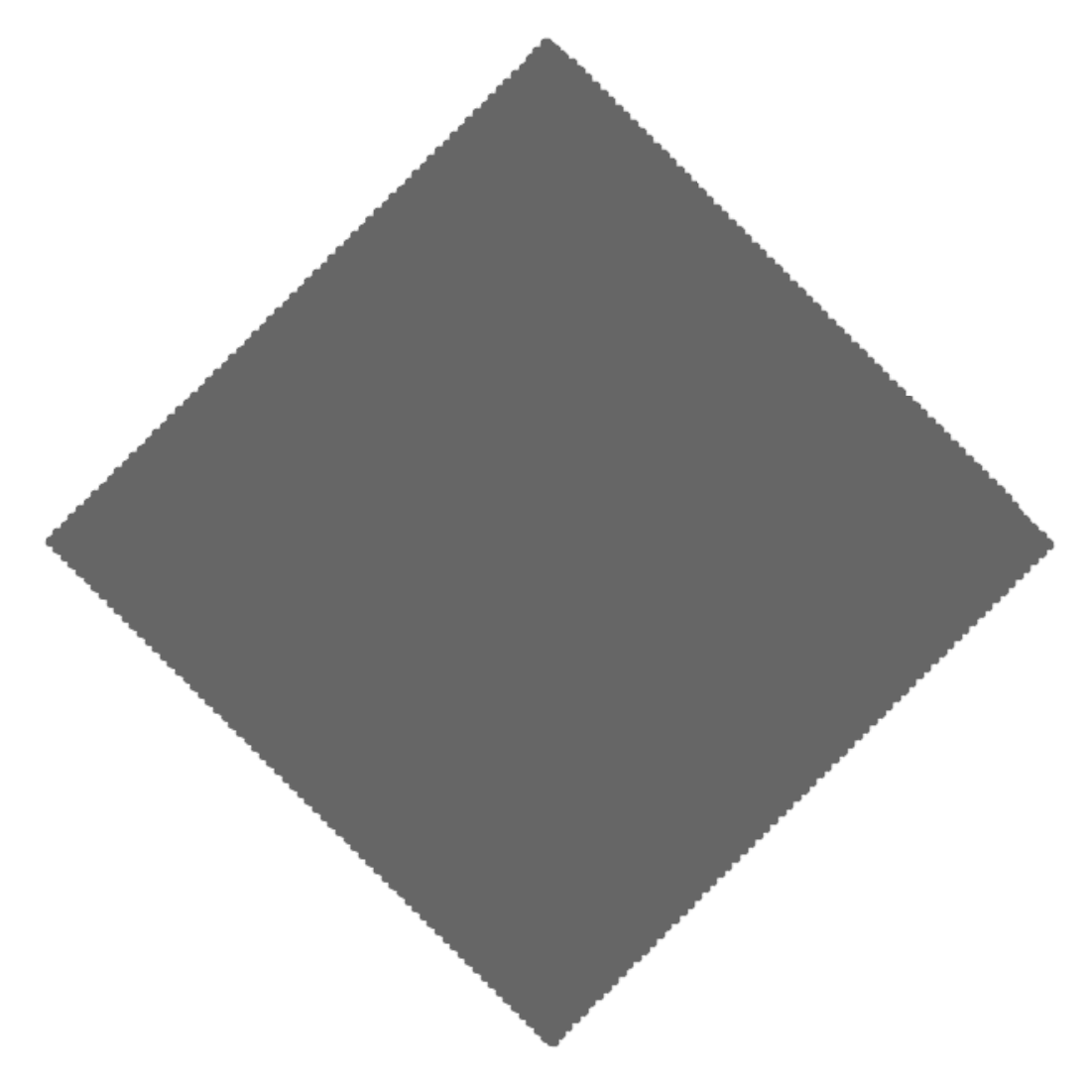};
        \end{axis}   
     \end{tikzpicture} 
           &
      \begin{tikzpicture} % alpha = 1, beta = 1/4 
        \begin{axis}[
          axis equal image,
          height = 1.125\figureheight,
          major tick length = \ticklength,
          axis on top,
          xtick={-0.6,0,0.6}, ytick={-0.5,0,0.5}, 
          xmin = -0.9, xmax = 0.9, ymin = -0.6, ymax = 0.6, 
          x axis line style = {-}, y axis line style = {-}
          ]
          \addplot graphics
                   [xmin=-0.63,xmax=0.23,ymin=0.1,ymax=0.4]
                   {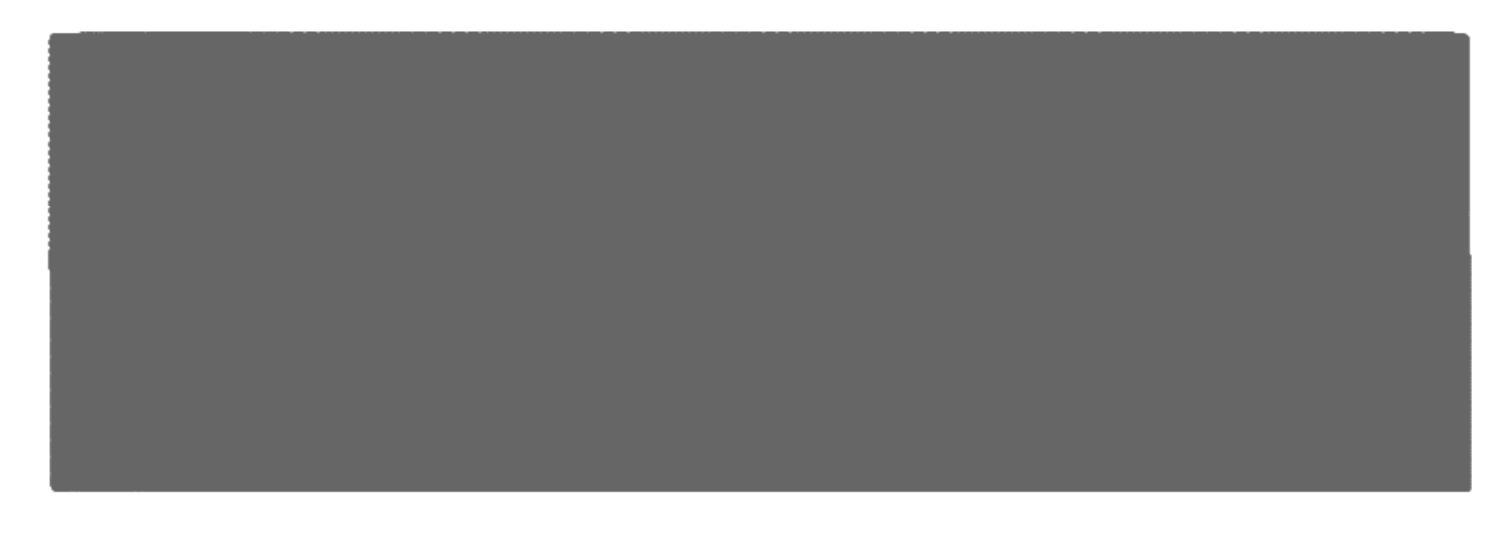};
        \end{axis}   
      \end{tikzpicture}
&
      \begin{tikzpicture} % alpha = 1, beta = 1/4 
        \begin{axis}[
          axis equal image,
          height = 1.298\figureheight,
          major tick length = \ticklength,
          axis on top,
          xtick={-1,-0.5,0,0.5,1}, ytick={-0.5,0,0.5}, 
          xmin = -1.1, xmax = 1.1, ymin = -0.6, ymax = 0.6,        
          x axis line style = {-}, y axis line style = {-},
          ]
          \addplot graphics
                   [xmin=-1.05,xmax=1.05,ymin=-0.3,ymax=0.3]
                   {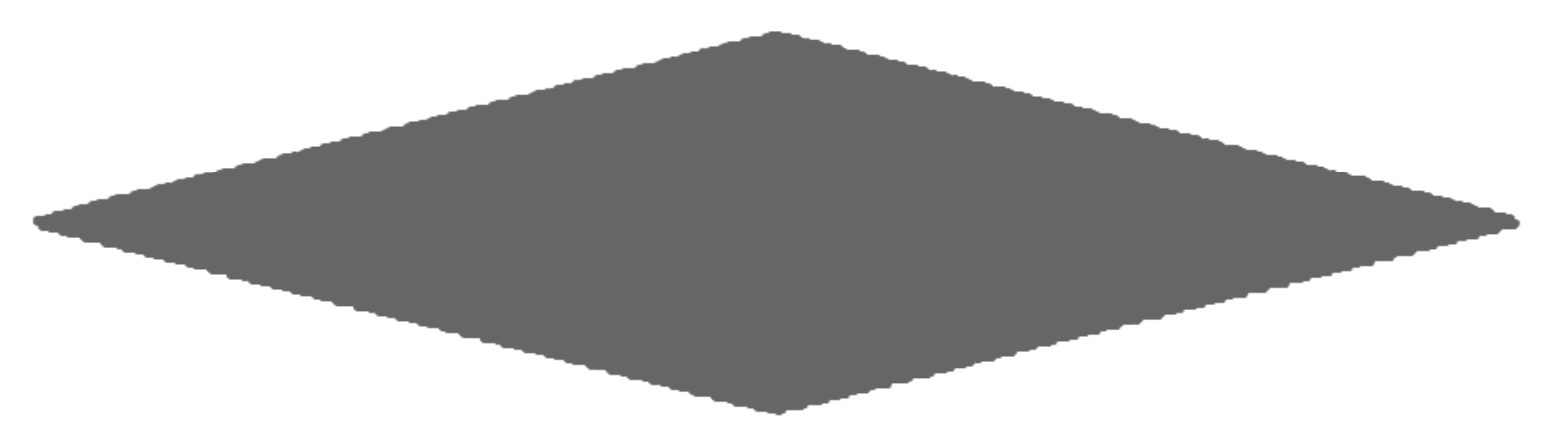};
        \end{axis}   
      \end{tikzpicture}    
      \\
     \begin{tikzpicture} % alpha = beta = 0.6
        \begin{axis}[
          height = 1.58\figureheight,
          axis equal image,
          xmin=-0.6,xmax=0.6,ymin=-1.2,ymax=1.2,
          axis on top,
          major tick length = \ticklength,
          xtick={-0.5,0,0.5}, ytick={-1,-0.5,0,0.5,1},          
          x axis line style = {-}, y axis line style = {-}
          ]
          \addplot graphics
                    [xmin=-0.55,xmax=.55,ymin=-0.55,ymax=0.55]
                   {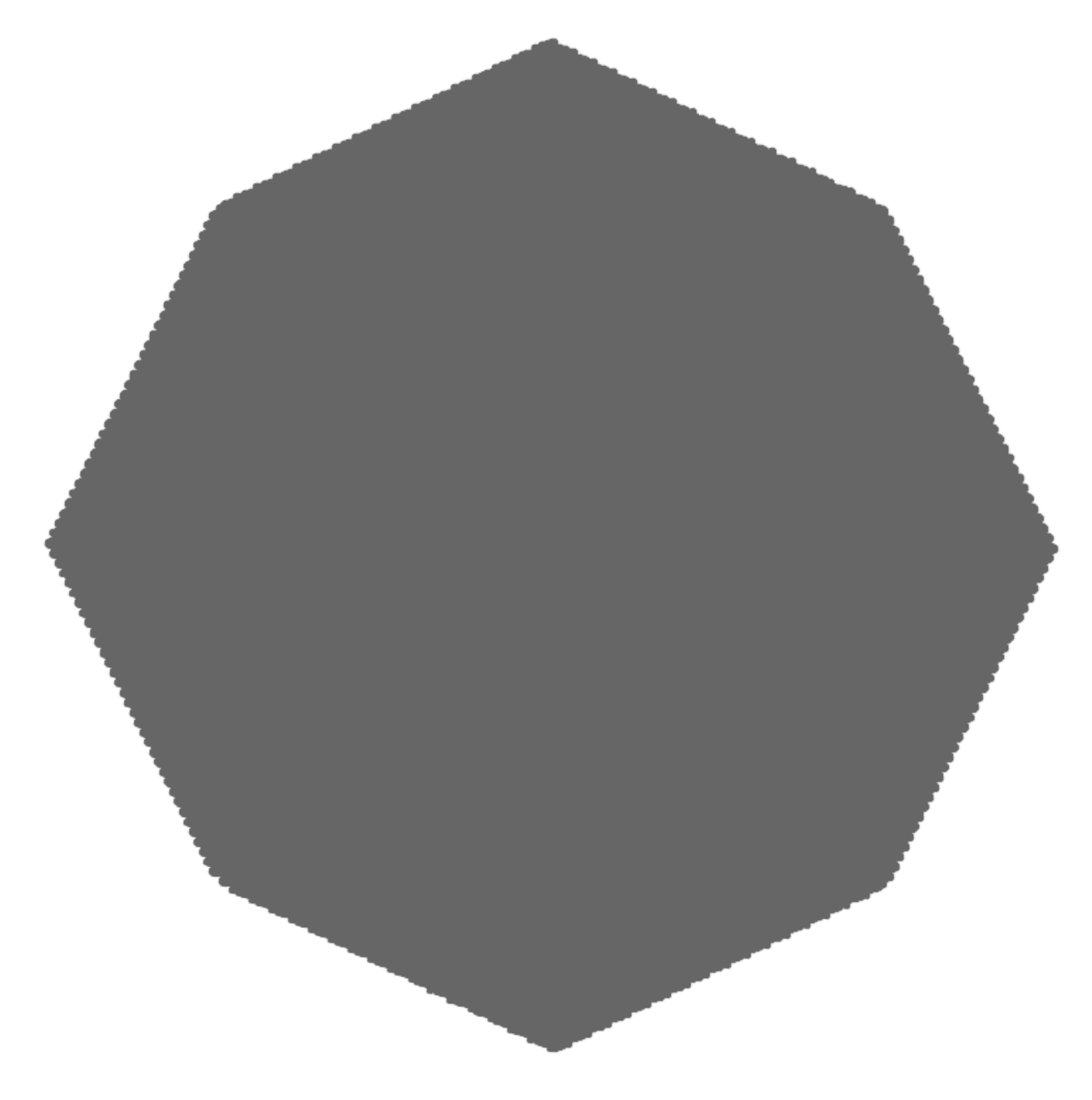};
        \end{axis}   
     \end{tikzpicture}    
        &    
     \begin{tikzpicture} % alpha = 3/4, beta = 1
         \begin{axis}[
          height = 1.601\figureheight,
          axis equal image,
          xmin=-0.9,xmax=0.9,ymin=-1.2,ymax=1.2,
          axis on top,
          major tick length = \ticklength,
          xtick={-1,-0.5,0,0.5,1}, ytick={-1,-0.5,0,0.5,1},      
          x axis line style = {-}, y axis line style = {-},
          ]
          \addplot graphics
                   [xmin=-0.85,xmax=0.85,ymin=-1.1,ymax=1.1]
                   {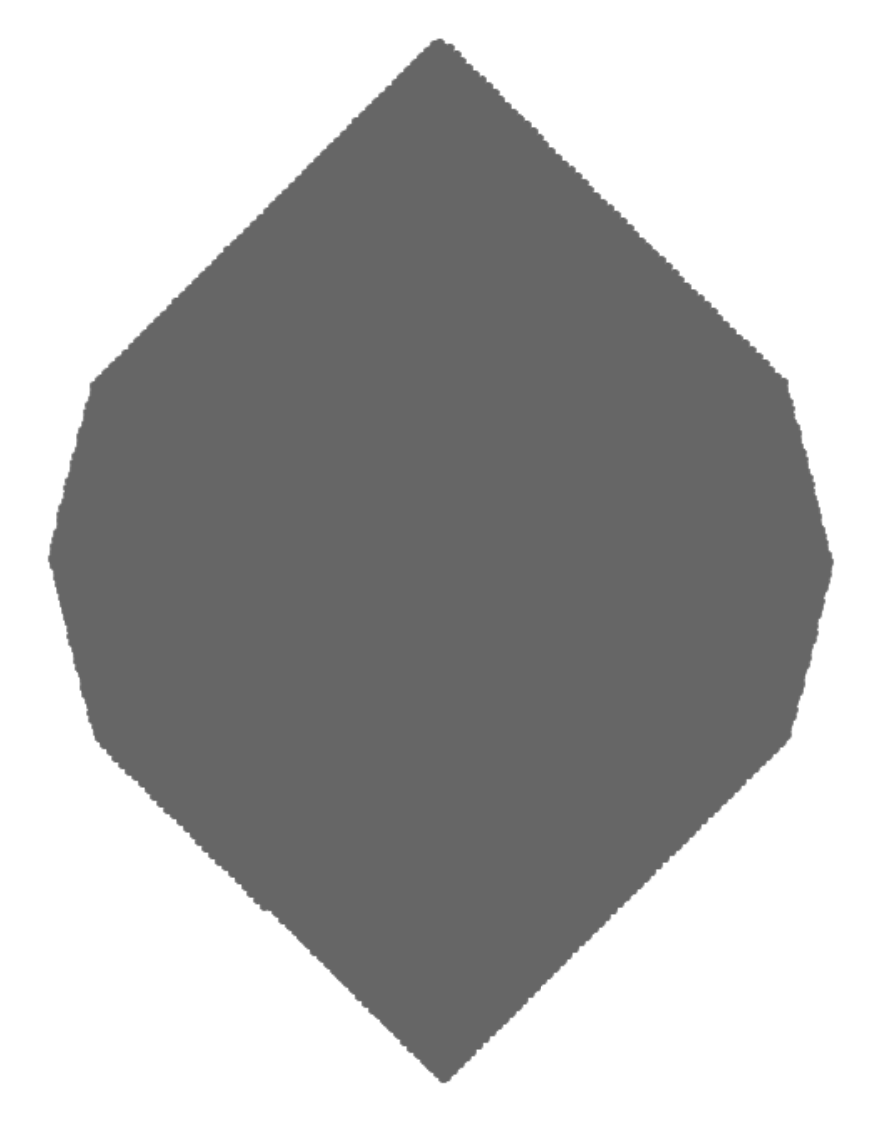};
        \end{axis}   
  \end{tikzpicture}   
 &  
     \begin{tikzpicture} % alpha = beta = 1
         \begin{axis}[
          height = 1.601\figureheight,
          axis equal image,
          xmin=-1.1,xmax=1.1,ymin=-1.2,ymax=1.2,
          axis on top,
          major tick length = \ticklength,
          xtick={-1,-0.5,0,0.5,1}, ytick={-1,-0.5,0,0.5,1},      
          x axis line style = {-}, y axis line style = {-},
          ]
          \addplot graphics
                   [xmin=-1.05,xmax=1.05,ymin=-1.05,ymax=1.05]
                   {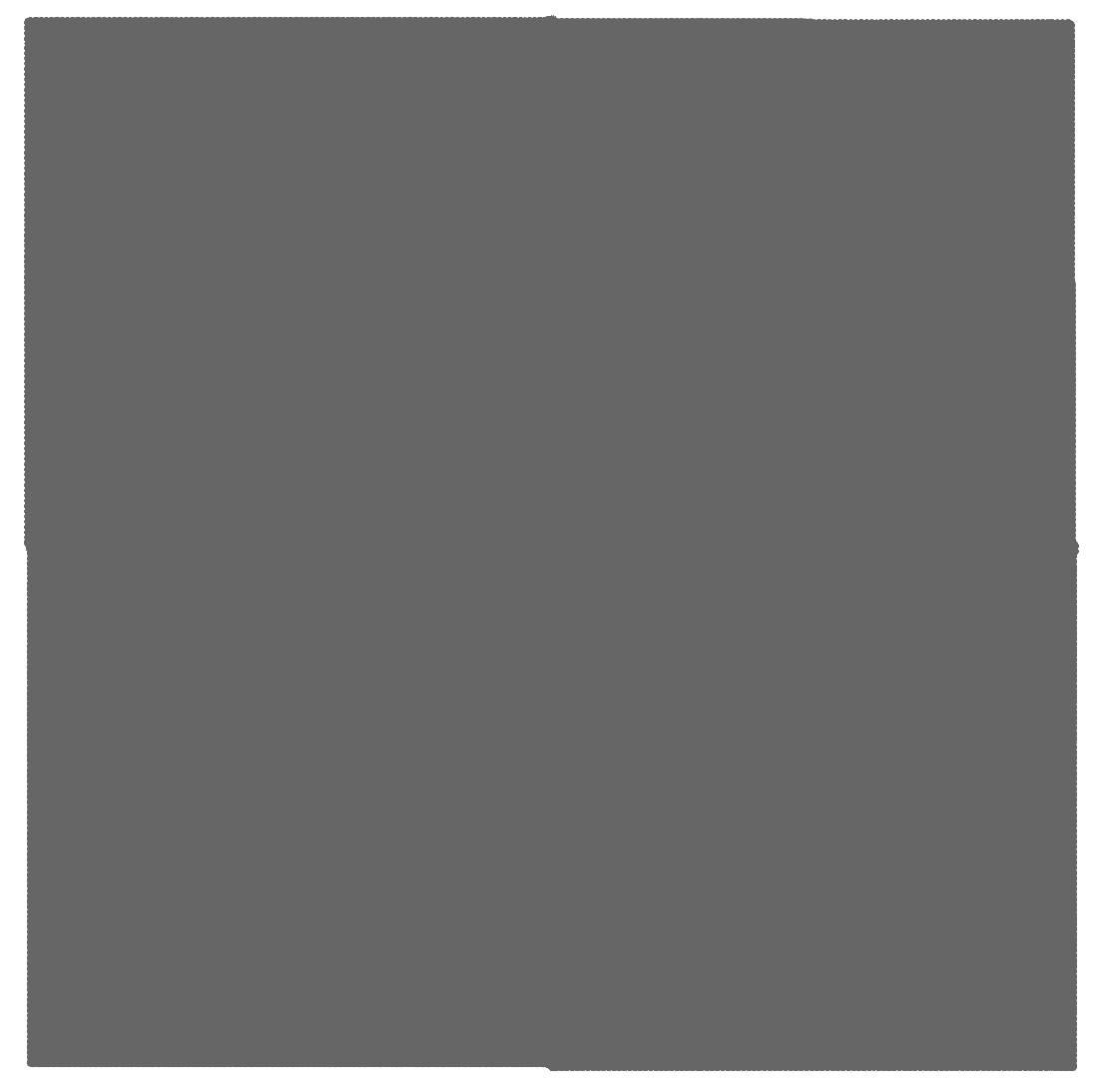};
        \end{axis}   
  \end{tikzpicture}       
  \end{tabular}
   \end{center}   
%   \caption{Approximate rotation sets $Q_{k,n}^*$ of the perturbed maps $\bar{F}_{\nicefrac{1}{2},\nicefrac{1}{2}}, \bar{G}, \bar{F}_{1,\nicefrac{1}{4}}, \bar{F}_{\nicefrac{3}{5},\nicefrac{3}{5}}$, $\bar{F}_{\nicefrac{3}{4},1}$ and $\bar{F}_{1,1}$ (with $k=50,60,16,50,45,8$,
%     $n=130,130,140,100,80,100$, $r_1 = 0.012,0.008, 0.012,0.01,0.002,0.022$,
%     and $r_2=0.014,0.001,0.002,0.011,0.013,0.015$)}
  \caption{Approximations $Q_{k,n}^*$ for the rotation sets of the perturbed maps $\bar{F}_{\nicefrac{1}{2},\nicefrac{1}{2}}, \bar{G}, \bar{F}_{1,\nicefrac{1}{4}}, \bar{F}_{\nicefrac{3}{5},\nicefrac{3}{5}}$, $\bar{F}_{\nicefrac{3}{4},1}$ and $\bar{F}_{1,1}$ (from top left to bottom right) according to Table \ref{t.table}.}
  \label{f.further_RS_pert}   
\end{figure}

\section{Conclusion}

\noindent
In conclusion, the set-oriented approach to the computation of rotation sets
provides better and more stable results than conventional direct
approaches. Moreover, it can at least partially be backed up by rigorous
convergence results, even if the theoretical error estimates are not useful in
practice. The much better performance of the algorithm for specific examples
finds a possible explanation in the likely presence of a shadowing property,
which can again be backed up by rigorous results.

What remains is to use this new numerical method in order to perform a
systematic and detailed study of the behaviour and bifurcations of rotation sets
in standard parameter families, as the one given by
(\ref{e.standard_family}). Of course, it is highly likely that the performance
of the algorithm becomes increasingly worse as bifurcation parameters are
approached, at which the rotation set changes and structural stability and
shadowing therefore have to break down. Therefore, it seems feasible to carry
out such investigations in collaboration with experts on scientific computing
and access to high-performance computing facilities, so that at least the limits
of contemporary computing capacities can be exhausted to partially counter these
effects. We leave this as a task for future research.\medskip

\begin{center}
  {\bf Matlab codes will be made available on the authors' homepages.}
\end{center}

% \bibliography{dynamics}
%\bibliographystyle{unsrt}

\end{document}

%% file: bibstandard.tex
\newcommand{\eqand}{\ensuremath{\quad \textrm{and} \quad}}

\newcommand{\ld}{\ensuremath{,\ldots,}}
\newcommand{\ssq}{\ensuremath{\subseteq}}

\newcommand{\eps}{\ensuremath{\varepsilon}}

\newcommand{\Leb}{\ensuremath{\mathrm{Leb}}}

\newcommand{\diam}{\ensuremath{\mathrm{diam}}}

\newcommand{\conv}{\ensuremath{\mathrm{Conv}}}

\newcommand{\torus}{\ensuremath{\mathbb{T}^2}}

\newcommand{\nfolge}[1]{\ensuremath{(#1)_{n\in\mathbb{N}}}}

\newcommand{\alphlist}{\begin{list}{(\alph{enumi})}{\usecounter{enumi}\setlength{\parsep}{2pt}
      \setlength{\itemsep}{1pt} \setlength{\topsep}{5pt}
      \setlength{\partopsep}{3pt}}}
\newcommand{\arablist}{\begin{list}{(\arabic{enumi})}{\usecounter{enumi}\setlength{\parsep}{2pt}
          \setlength{\itemsep}{1pt} \setlength{\topsep}{5pt}
          \setlength{\partopsep}{3pt}}}
\newcommand{\romanlist}{\begin{list}{(\roman{enumi})}{\usecounter{enumi}\setlength{\parsep}{2pt}
              \setlength{\itemsep}{1pt} \setlength{\topsep}{5pt}
              \setlength{\partopsep}{3pt}}}
\newcommand{\Romanlist}{\begin{list}{(\Roman{enumi})}{\usecounter{enumi}\setlength{\parsep}{2pt}
              \setlength{\itemsep}{1pt} \setlength{\topsep}{5pt}
              \setlength{\partopsep}{3pt}}}
\newcommand{\bulletlist}{\begin{list}{$\bullet$}{\setlength{\parsep}{2pt}
                \setlength{\itemsep}{1pt} \setlength{\topsep}{5pt}
                \setlength{\partopsep}{3pt}\setlength{\leftmargin}{15pt}}} 
\newcommand{\Alphlist}{\begin{list}{(\Alph{enumi})}{\usecounter{enumi}\setlength{\parsep}{2pt}
      \setlength{\itemsep}{1pt} \setlength{\topsep}{5pt}
      \setlength{\partopsep}{3pt}}}
 \newcommand{\listend}{\end{list}}

\newcommand{\T}{\ensuremath{\mathbb{T}}}

\newcommand{\N}{\ensuremath{\mathbb{N}}} 
\newcommand{\R}{\ensuremath{\mathbb{R}}}
\newcommand{\Z}{\ensuremath{\mathbb{Z}}}

\newcommand{\cB}{\mathcal{B}}
\newcommand{\cC}{\mathcal{C}}

\newcommand{\cI}{\mathcal{I}}

\newcommand{\nLim}{\ensuremath{\lim_{n\rightarrow\infty}}}
\newcommand{\iLim}{\ensuremath{\lim_{i\rightarrow\infty}}}

\newcommand{\inergsum}{\ensuremath{\sum_{i=0}^{n-1}}}

\newcommand{\ntel}{\ensuremath{\frac{1}{n}}}